\pdfoutput=1
\documentclass{article}

\usepackage{PRIMEarxiv}

\usepackage[utf8]{inputenc} 
\usepackage[T1]{fontenc}    
\usepackage{hyperref}       
\usepackage{url}            
\usepackage{booktabs}       
\usepackage{amsfonts}       
\usepackage{nicefrac}       
\usepackage{microtype}      
\usepackage{xcolor}         
\usepackage{fancyhdr}       
\usepackage{float}           
\usepackage{graphicx}       
\graphicspath{{Figures/}{media/}}
\usepackage{subcaption}
\usepackage{amsmath}
\usepackage{amsthm} 
\theoremstyle{plain}
\newtheorem{theorem}{Theorem}[section]
\newtheorem{proposition}[theorem]{Proposition}
\newtheorem{lemma}[theorem]{Lemma}
\newtheorem{hypothesis}[theorem]{Hypothesis}

\theoremstyle{definition}
\newtheorem{definition}[theorem]{Definition}
\newtheorem{remark}[theorem]{Remark}
\title{Multilevel Isogeometric Projection Stabilization via Quasi-Interpolation for Advection-Dominated Problems

}

\author{
  Zakaria El Hasnaoui \\
  Vanguard Center \\
  Mohammed VI Polytechnic University (UM6P) \\
  Rabat, Morocco\\
  \texttt{zakaria.elhasnaoui@um6p.ma} \\
   \And
  Ahmed Ratnani \\
  Vanguard Center \\
  Mohammed VI Polytechnic University (UM6P) \\
  Rabat, Morocco\\
  \texttt{ahmed.ratnani@um6p.ma} \\
}

\begin{document}
\maketitle

\begin{abstract}
This paper presents a novel multilevel projection-based stabilization method for advection-dominated convection--diffusion problems within the framework of Isogeometric Analysis. The proposed approach extracts and penalizes fine-scale fluctuations using continuous B-spline quasi-interpolants, avoiding both the highly sensitive parameters used in residual-based stabilization methods and the discontinuous auxiliary spaces required by classical Local Projection Stabilization. Stabilization is applied hierarchically across nested levels of the discrete space via explicit mesh-dependent weights. We establish the theoretical foundation of the method by deriving \emph{a priori} error estimates, supplemented by a discrete inf-sup condition established for the one-dimensional setting with constant advection under a numerically validated stability hypothesis that ensures robust streamline derivative control. Numerical experiments on stringent benchmarks demonstrate the method's ability to significantly reduce spurious oscillations across a variety of regimes, including the limiting cases of pure advection and advection--reaction. Notably, despite being a fully linear formulation, the method achieves significant reduction of undershoots near sharp layers, providing an effective stabilization mechanism without requiring complex nonlinear terms. Furthermore, by utilizing a robust global parameter scaling, the proposed approach significantly alleviates the parameter sensitivity that typically affects residual-based alternatives, reducing the strong dependence on problem-dependent tuning.
\end{abstract}


\section{Introduction}

Transport equations, possibly including perturbation terms such as diffusion or reaction, appear deceptively simple at the continuous level. Nevertheless, they play a central role in many physical applications and serve as fundamental model problems for the analysis of nonlinear and more complex flow phenomena. Despite their simple structure, obtaining stable and accurate numerical approximations remains a challenging task \cite{john_finite_2018,Openproblems_ROss}.

A prototypical example is the advection–diffusion–reaction equation. It is well known that the standard Galerkin finite element method exhibits severe instabilities when applied to advection-dominated or advection–reaction problems. These instabilities typically manifest as non-physical oscillations in the numerical solution \cite{HARARI1994165}.

Over the past decades, numerous stabilization techniques have been proposed to overcome these difficulties \cite{hughes_multiscale_2017}. A comprehensive historical overview of the development of stabilized methods can be found in \cite{AUGUSTIN20113395,Hauke-Sangalli-Doweidar}. Broadly speaking, stabilized methods can be classified into two main categories: (1) residual-based stabilization methods, such as Variational Multiscale (VMS), Streamline Upwind Petrov-Galerkin (SUPG), and Galerkin Least-Squares (GLS) \cite{TENEIKELDER2018259,hughes_multiscale_2017}, and (2) penalty or projection-based methods, including Orthogonal Subscales (OSS), Subgrid Viscosity (SVG), and Local Projection Stabilization (LPS) techniques \cite{Braack_Burman,guermond_subgrid_1999,codina_stabilized_2000}.

Residual-based methods are strongly consistent and naturally allow for high-order discretizations. However, they may suffer from drawbacks related to the non-symmetric structure of the stabilization terms and the presence of a tuning parameter. These aspects can complicate the analysis of those methods.

For these reasons, in the present work we adopt the second strategy, focusing on projection-based stabilization methods.

The OSS stabilization strategy provides a systematic way to design high-order stabilized methods. This approach was introduced in \cite{codina_stabilized_2000}, where the classical Galerkin formulation is enriched by adding a projection-based penalty term acting on selected components, typically the gradient or residual terms. However, a major drawback of this classical approach is that the required global $L^2$ projection operator is often computationally expensive to evaluate.

LPS methods overcome this computational difficulty by using element-wise local projections satisfying suitable orthogonality properties, instead of a global projection operator. In these approaches, the effect of unresolved scales on the small resolved scales is modeled through artificial diffusion terms applied locally. Furthermore, LPS methods belong to the class of symmetric stabilization techniques \cite{knobloch_local_2009}. A notable advantage of symmetric stabilization, especially in PDE-constrained optimization problems, is that discretization and optimization commute \cite{knobloch_error_2019}, which simplifies both the theoretical analysis and the numerical implementation.

Classical stabilization techniques such as SUPG and GLS were originally developed within the finite element framework and have been widely and successfully applied to advection- and advection–diffusion-dominated problems. With the emergence of Isogeometric Analysis (IGA), several studies have investigated the integration of SUPG-type stabilization within spline-based discretizations \cite{HUGHES20054135,bazilevs_isogeometric_2006}.

However, it has been consistently observed that merely increasing the approximation order and continuity does not automatically suppress spurious oscillations. In particular, Gibbs-type oscillations, commonly encountered in high-order finite element discretizations, tend to intensify as the polynomial degree and smoothness increase, even when SUPG stabilization is employed \cite{MANNI2011511}.

More generally, numerical investigations indicate that classical stabilized formulations, including SUPG and GLS, perform satisfactorily for convection- or advection–diffusion-dominated problems in the absence of sharp internal or boundary layers. Their performance, however, deteriorates significantly in the presence of steep gradients, where oscillations, excessive smearing, and sensitivity to stabilization parameters frequently arise \cite{JOHN201443}. The tuning of stabilization parameters remains a delicate issue, particularly in high-order and high-regularity discretizations.

Recent developments within IGA have therefore focused on modifying or enhancing classical SUPG-type methods. Several approaches augment the residual weighting or reformulate stabilization terms in order to mitigate known shortcomings in advection– and advection–diffusion–reaction-dominated regimes. In addition, adaptations of the spline approximation space itself such as exponential splines or variable-degree spline constructions have been proposed to better capture sharp gradients and reduce oscillatory behavior.

For instance, the GSC methods introduced in \cite{key_finite_2023} extend stabilization concepts derived from one-dimensional analysis by applying SUPG-type weighting to the advective operator, complemented by residual-based gradient weightings associated with reaction terms. These modifications improve upon classical SUPG by reducing overshoots and undershoots in boundary layer regions, particularly for high-order Lagrange and B-spline discretizations. Nevertheless, as with many residual-based approaches, these methods remain sensitive to stabilization parameter choices, requiring careful tuning to balance accuracy and artificial diffusion.

Similarly, IGA formulations employing exponential or variable-degree spline spaces can substantially reduce oscillations in advection–diffusion problems due to their enhanced capacity to approximate sharp gradients \cite{MANNI2011511}. However, such approaches modify the underlying approximation space, introducing additional problem-dependent parameters and nonstandard basis functions, which may limit general applicability and complicate integration into standard IGA frameworks.

To place the present contribution in its proper context, we now briefly recall the main categories of stabilization methods and clarify the fundamental structural differences with the proposed approach.

\medskip
\noindent
\textbf{(i) Residual-based methods (SUPG/GLS/VMS).}
In Streamline Upwind Petrov--Galerkin (SUPG) and Galerkin Least-Squares (GLS) methods~\cite{brooks_streamline_1982}, and their Variational Multiscale (VMS) generalization~\cite{HUGHES19983}, the stabilization is achieved by testing the discrete equation with a modified test function of the form $v_h + \tau\, \mathcal{L}(v_h)$, where $\mathcal{L}$ is the differential operator (or the streamline derivative $\mathbf{b}\cdot\nabla v_h$). This requires computing the \emph{strong residual} of the PDE, which involves second-order derivatives $\Delta u_h$. While IGA B-splines possess the required inter-element continuity, a well-recognized practical challenge of residual-based methods is the identification and tuning of the stabilization parameter $\tau$, which is sensitive to the problem configuration and for which a robust, universally accepted definition for high-order elements remains elusive. Additionally, these residual evaluations increase the complexity of the formulation and degrade algebraic conditioning. These methods are strongly consistent, but are \emph{non-symmetric}: $s(u_h, v_h) \neq s(v_h, u_h)$.

\medskip
\noindent
\textbf{(ii) LPS.}
LPS methods~\cite{knobloch_local_2009,matthies_unified_2007} achieve symmetric stabilization by penalizing the fluctuation of the \emph{streamline derivative} $\mathbf{b}\cdot\nabla u_h$ that is not captured by a coarse space. The stabilization has the general form
\[
s_{\mathrm{LPS}}(u_h, v_h) = \sum_{M} \tau_M \int_M \kappa_M(\mathbf{b}\cdot\nabla u_h)\, \kappa_M(\mathbf{b}\cdot\nabla v_h)\, dx,
\]
where $\kappa_M = I - \pi_M$ and $\pi_M$ is a local $L^2$ projection onto a \emph{discontinuous} polynomial space of degree $p-1$ defined on macro-elements $M$. This approach avoids computing the residual, but the construction of the discontinuous auxiliary space, the element-wise $L^2$ projection, and the macro-element infrastructure make it difficult to adapt to the highly smooth $C^{p-1}$ setting of Isogeometric Analysis.

\medskip
\noindent
\textbf{(iii) Orthogonal Subscales (OSS).}
The Orthogonal Subscales method, introduced by Codina~\cite{codina_stabilized_2000}, provides a symmetric stabilization within the Variational Multiscale framework by assuming that the unresolved subscales are $L^2$-orthogonal to the finite element space $V_h$. In its term-by-term formulation for the convective term, the stabilization takes the form
\[
s_{\mathrm{OSS}}(u_h, v_h) = \sum_{K} \tau_K \left(\Pi_h^\perp(\mathbf{b} \cdot \nabla u_h),\; \Pi_h^\perp(\mathbf{b} \cdot \nabla v_h)\right)_K,
\]
where $\Pi_h^\perp = I - \Pi_h$ is the $L^2$-orthogonal fluctuation operator and $\Pi_h : L^2(\Omega) \to V_h$ is the global $L^2$-orthogonal projection. Unlike LPS, the projection target is $V_h$ itself (a continuous space), so OSS is in principle compatible with the smooth B-spline setting of Isogeometric Analysis. However, the $L^2$ projection $\Pi_h$ requires solving the mass matrix system $M \boldsymbol{\lambda} = \mathbf{g}$ at each evaluation. For $C^{p-1}$ B-splines, the mass matrix is banded but \emph{not diagonal}, and mass lumping is known to be problematic for high-order, high-continuity splines. The projection is therefore a \emph{global} operation that couples all degrees of freedom through $M^{-1}$. Furthermore, OSS applies the fluctuation operator to $\mathbf{b} \cdot \nabla u_h$, which is a function that \emph{does not belong to $V_h$} (the derivative reduces the spline degree and the multiplication by a variable $\mathbf{b}$ further degrades regularity). The $L^2$ projection of this less regular function back onto the smooth $C^{p-1}$ space may introduce projection artifacts, particularly when the advection field exhibits sharp spatial variations.

\medskip
\noindent
\textbf{(iv) Subgrid viscosity methods (Guermond).}
Guermond's subgrid viscosity method~\cite{guermond_subgrid_1999} decomposes the discrete solution into two components: a \emph{resolved scale} $P_H u_h$ captured by a coarse finite element space $V_H \subset V_h$, and a \emph{subgrid scale} $(I - P_H) u_h$ representing the unresolved high-frequency content. Artificial viscosity is then added proportionally to the subgrid fluctuation, acting on the \emph{full gradient} $\nabla u_h$. A crucial feature of this approach is that the error analysis relies on the \emph{approximation properties of the coarse space} $V_H$ itself, not on those of the projector $P_H$. For this reason, an \emph{oblique projection} suffices: $P_H$ only needs to satisfy $P_H|_{V_H} = \mathrm{Id}$ (i.e., reproduce the coarse space), but it does not need to be an $L^2$-orthogonal projector or possess optimal approximation estimates. The oblique projector merely serves to \emph{separate scales}, not to approximate.

It is worth emphasizing the fundamental role played by the \emph{nature of the projection operator} in the design of stabilization methods. Four structurally distinct projectors appear in the literature:
\begin{enumerate}
\item \textbf{$L^2$-orthogonal projection} $\pi_M$ onto a discontinuous space (LPS): the projector's approximation properties directly enter the consistency error bound; requires solving local mass matrices on macro-elements;
\item \textbf{Global $L^2$-orthogonal projection} $\Pi_h$ onto $V_h$ itself (OSS): the fluctuation operator is applied to the streamline derivative $\mathbf{b} \cdot \nabla u_h$; compatible with IGA but requires a global mass matrix solve;
\item \textbf{Oblique projection} $P_H$ onto a coarse space (Guermond): only separates resolved and subgrid scales; the coarse space controls the error, so no approximation properties are needed from the projector itself;
\item \textbf{Quasi-interpolant} $\Pi_k^p$ onto a continuous B-spline space (present work): an approximation projector that reproduces polynomials of degree $p$, has local support, requires no linear system solve, and enjoys optimal approximation estimates. As in LPS, the projector's approximation order directly controls the consistency error.
\end{enumerate}
The use of a quasi-interpolant as the projection operator is one of the key innovations of the proposed method, as it simultaneously provides optimal approximation order, computational efficiency (no mass matrix inversion), and full compatibility with the smooth B-spline structure of Isogeometric Analysis.

\medskip
\noindent
\textbf{(v) The proposed method.}
The present work introduces a fundamentally different stabilization mechanism that is native to the Isogeometric Analysis framework. The key structural distinctions are:

\begin{itemize}
\item \textbf{Projection of the primal variable.} Unlike both LPS and OSS, which apply the projection operator to the streamline derivative $\mathbf{b}\cdot\nabla u_h$ (a function that generally does not belong to $V_h$), our method applies a quasi-interpolant operator $\Pi_k^p$ to the primal variable $u_h$ \emph{itself}, and then differentiates. This ``project-then-differentiate'' design ensures that the projection operates exclusively on smooth B-spline functions within $V_h$, avoiding the need to project quantities of reduced regularity back onto the discrete space. The fluctuation $u_h - \Pi_k^p u_h$ captures the fine-scale content of $u_h$ that the coarse mesh $\mathcal{M}_h^{(k)}$ cannot resolve.

\item \textbf{Continuous B-spline projection space.} LPS relies on a local $L^2$ projection onto a \emph{discontinuous} polynomial space, which is foreign to IGA. Our method uses a quasi-interpolant onto a \emph{continuous} B-spline space of the same degree $p$, fully exploiting the smooth structure of IGA discretizations.

\item \textbf{Multilevel hierarchy.} Classical LPS, VMS, and subgrid viscosity are all two-level methods: they decompose the solution into resolved and unresolved scales using a single coarse mesh. Our method introduces a true multilevel hierarchy $\mathcal{M}_h^{(1)}, \dots, \mathcal{M}_h^{(L)}$, with each level contributing to the stabilization through its own quasi-interpolant $\Pi_k^p$ and mesh-dependent weight $c_k = h/H^{(k)}$. This multilevel structure provides richer scale decomposition while maintaining optimal convergence.
\end{itemize}

In summary, the proposed method combines the symmetry advantage of projection-based stabilization with the smooth, high-order infrastructure of Isogeometric Analysis, while introducing a novel multilevel scale decomposition that has no counterpart in existing LPS, VMS, or subgrid viscosity frameworks.

The stability of the method is established through coercivity in a natural norm (MQ norm) with a priori error estimates, as well as through a discrete inf-sup condition proven for a specialized one-dimensional setting with a constant advection field under a numerically validated hypothesis that demonstrates explicit control of the full streamline derivative (MQSD norm). Moreover, we provide extensive numerical evidence demonstrating that the proposed stabilization effectively suppresses spurious oscillations in challenging advection-dominated and advection--reaction-dominated benchmarks.

\medskip
\noindent
\textbf{Outline.}
The remainder of this paper is organized as follows. Section~2 introduces the model advection--diffusion--reaction problem, the associated variational formulation, and the standard assumptions on the problem data. Section~3 recalls the construction of univariate and multivariate B-spline spaces and presents the quasi-interpolant operators, together with the key approximation and inverse estimates used throughout the analysis. In Section~4, we formulate the multilevel projection stabilization, define the stabilized bilinear form and the associated energy norms, and establish the general stability and \emph{a priori} error analysis in the MQ norm. Section~5 is devoted to the enhanced inf-sup stability in the MQSD norm: we state the underlying stability hypothesis, validate it numerically via the Bressan--Sangalli dimension condition, and prove the discrete inf-sup condition for the one-dimensional constant-advection setting. Section~6 presents an extensive suite of numerical experiments including advection--reaction, advection--diffusion, parabolic boundary layers, internal layers, rotational flow, and pure advection benchmarks that demonstrate the robustness and accuracy of the proposed method. Finally, Section~7 summarizes the main contributions and discusses open problems and future research directions.

\section{Model Problem and Assumptions}

Let $\Omega \subset \mathbb{R}^d$, $d \ge 1$, be a bounded domain with polyhedral Lipschitz boundary $\partial \Omega$. 
We study the convection--diffusion--reaction problem
\begin{equation}
\label{ADR_equa}
-\varepsilon \Delta u + \mathbf{b} \cdot \nabla u + c\, u = f \quad \text{in } \Omega, 
\qquad 
u = 0 \quad \text{on } \partial \Omega,
\end{equation}
where $\varepsilon>0$ is the diffusion coefficient, $\mathbf{b}$ is the convection field, $c$ is the reaction coefficient, and $f$ is a given source term.  

We assume the data satisfy
\begin{equation*}
\mathbf{b} \in W^{1,\infty}(\Omega)^d, \quad 
c \in L^\infty(\Omega), \quad
f \in L^2(\Omega),
\end{equation*}
and that the reaction term is coercive, i.e.,
\begin{equation}
\sigma := c - \frac{1}{2} \operatorname{div} \mathbf{b} \ge \sigma_0 > 0,
\end{equation}
for some constant $\sigma_0$.  
Under these conditions, problem \eqref{ADR_equa} admits a unique solution $u \in H^1(\Omega)$.

Then, the weak solution of problem \eqref{ADR_equa} is a function $u \in H^1_0(\Omega)$ such that 

\begin{equation}    
\label{VF_equa}
a(u, v)=(f, v) \quad \forall v \in H^1_0(\Omega),
\end{equation}

where
\begin{equation}  
a(u, v)=\varepsilon(\nabla u, \nabla v)+(\boldsymbol{b} \cdot \nabla u, v)+(c u, v) .
\end{equation}

In this work, we focus on the numerical approximation of \eqref{VF_equa}. Accordingly, let \(V_h\) be a finite-dimensional approximation space of \(H_0^1(\Omega)\). 

The discrete variational scheme for \eqref{VF_equa} reads as follows: Find \(u_h \in V_h\), such that
\begin{equation}
\label{DVF_equa}
a(u_h, v_h) = (f, v_h), \quad \forall v_h \in V_h.
\end{equation}
It is well known (see \cite{key_finite_2023,knobloch_local_2009,hauke_exploring_2004}) that the standard Galerkin finite element method is generally unsuitable for the numerical approximation of problem \eqref{DVF_equa},
as the discrete solution often exhibits nonphysical oscillations when convection dominates diffusion, i.e., when $\|\mathbf{b}\| \gg \varepsilon$.

In this work, we focus on stabilization using a projection-based method with B-splines, 
relying on a quasi-interpolant operator \cite{lyche_b-splines_2017,lee_examples_2001} in a multilevel setting. 
A similar approach was previously introduced for classical finite elements in \cite{knobloch_local_2009,he_two-level_2012,knobloch_generalization_2010,barrenechea_blending_2017,matthies_local_2015,matthies_unified_2007}, 
where a local projection onto a discontinuous coarse space of degree $p-1$ was employed. 
In that context, a tunable parameter, defined piecewise on each coarse element, was used, 
and the projection was typically taken as the local $L^2$ projection.  

In contrast, our method employs a multilevel quasi-interpolant operator onto a B-spline space of degree $p$ over a hierarchy of coarse meshes. 
In the following sections, we introduce the B-spline basis and the construction of the quasi-interpolant operator in detail.

\section{B-splines and Quasi-Interpolant}
\subsection{B-splines and discrete spaces}

Let $T=(t_0, t_1, \dots, t_m)$ be a non-decreasing sequence of real numbers, called a \emph{knot vector}. 
The $i$-th univariate $B$-spline of order $p \in \mathbb{N}$ is defined recursively by
\begin{equation*}
N_i^0(t) :=
\begin{cases}
1, & t_i \le t < t_{i+1},\\
0, & \text{otherwise},
\end{cases}
\end{equation*}
and for $p \ge 1$,
\begin{equation*}
N_i^p(t) := \frac{t-t_i}{t_{i+p}-t_i} N_i^{p-1}(t) + \frac{t_{i+p+1}-t}{t_{i+p+1}-t_{i+1}} N_{i+1}^{p-1}(t), \quad t_0 \le t \le t_m,
\end{equation*}
for $i = 0, \dots, n := m-p-1$, with the convention that a fraction with zero denominator is taken to be zero.

The $B$-spline basis functions satisfy the following properties:
\begin{itemize}
    \item \textbf{Positivity:} $N_i^p(t) \ge 0$ for all $t \in [t_0, t_m]$.
    \item \textbf{Compact support:} $N_i^p(t) = 0$ for $t \notin [t_i, t_{i+p+1})$.
    \item \textbf{Partition of unity:} $\sum_{i=0}^n N_i^p(t) = 1$ for all $t \in [t_0, t_m]$.
\end{itemize}

\begin{definition}[$B$-spline curve]
Given control points $(P_0, \dots, P_n) \subset \mathbb{R}^d$ and a knot vector $T$, 
the $B$-spline curve of degree $p$ is
\[
\mathcal{C}(t) = \sum_{i=0}^n N_i^p(t) P_i.
\]
\end{definition}

Important properties of $B$-spline curves include:
\begin{itemize}
    \item $\mathcal{C}$ is a piecewise polynomial curve.
    \item \textbf{Affine invariance:} applying an affine map $\phi$ to $\mathcal{C}$ corresponds to transforming the control points: $\phi(\mathcal{C}(t)) = \sum_i N_i^p(t) \phi(P_i)$.
    \item \textbf{Local support:} if $t \in [t_i, t_{i+1})$, then $\mathcal{C}(t)$ depends only on $(P_{i-p}, \dots, P_i)$ and lies in their convex hull.
    \item \textbf{Local modification:} changing a control point $P_i$ affects $\mathcal{C}(t)$ only in $[t_i, t_{i+p+1})$.
    \item $\mathcal{C}$ is $\mathcal{C}^\infty$ in each open interval $(t_i, t_{i+1})$, and $\mathcal{C}^{p-\mu_i}$ at knot $t_i$, where $\mu_i$ is the knot multiplicity.
    \item For an open knot vector
    \[
    t_0 = \cdots = t_p < t_{p+1} \le \dots \le t_{m-p-1} < t_{m-p} = \cdots = t_m,
    \]
    the curve interpolates the endpoints: $\mathcal{C}(t_0) = P_0$, $\mathcal{C}(t_m) = P_n$.
\end{itemize}

\subsubsection{Schoenberg space and multivariate generalization}

The univariate Schoenberg space of degree $p$ associated with $T$ is
\[
\mathcal{S}^p(T) = \operatorname{span} \{ N_i^p \mid 0 \le i \le n\}.
\]
Classical spline theory shows that $\mathcal{S}^p(T)$ coincides with the space of piecewise polynomials of degree $p$ with smoothness $\mu_i$ at each knot $t_i$.

For the multivariate case, let $d \in \mathbb{N}^*$ and a family of knot vectors $\mathcal{T} = \{T^1, \dots, T^d\}$. Assuming $\Omega = (0,1)^d$ and that each $T^l$ is open,
\[
T^l = \big(\underbrace{0,\dots,0}_{p_l+1}, t_{l,p_l+1}, \dots, t_{m_l-p_l-1}, \underbrace{1,\dots,1}_{p_l+1}\big), \quad 1 \le l \le d,
\]
the tensor-product Schoenberg space is
\[
\mathcal{S}^{\mathbf{p}}(\mathcal{T}) = \mathcal{S}^{p_1}(T^1) \otimes \cdots \otimes \mathcal{S}^{p_d}(T^d),
\]
where $\mathbf{p}=(p_1,\dots,p_d)$, with basis functions
\[
N_{\mathbf{i}}^{\mathbf{p}}(x) = N_{i_1}^{p_1}(x_1) \cdots N_{i_d}^{p_d}(x_d), \quad 0 \le i_l \le m_l - p_l -1, \quad 1 \le l \le d.
\]

The approximation space for the Galerkin method is then
\begin{equation}\label{eq:Vh_def}
V_h = \mathcal{S}^{\mathbf{p}}(\mathcal{T}) \subset H^1(\Omega),
\end{equation}
and for homogeneous Dirichlet boundary conditions on $\Gamma_D \subset \partial \Omega$, we define
\[
V_{h, \Gamma_D} = \{ u_h \in V_h \mid u_h|_{\Gamma_D} = 0 \}.
\] 

\subsection{Quasi-Interpolant}\label{sec:quasi_interpolant}
Quasi-interpolants based on B-splines provide local and computationally efficient approximations without solving global systems. 
The concept originates from Schoenberg \cite{schoenberg_contributions_1946} and was further developed by de Boor \cite{de_boor_practical_2001,de_boor_spline_1973} for uniform and non-uniform splines. 
In particular, Lyche ~\cite{lyche_local_1975,lee_examples_2001,lyche_b-splines_2017}  constructed quasi-interpolant operators that preserve high-order approximation properties in both univariate and tensor-product spline spaces. 
These operators exploit the local support and smoothness of B-splines while maintaining optimal convergence rates.

Given the B-spline basis functions $\{N_i^p\}_{i=0}^n$ associated with the knot vector $T$, 
a spline approximation of a given function $f$ is written as
\[
\mathcal{Q}f(t) := \sum_{i=0}^n \lambda_i(f)\, N_i^p(t),
\]
where $\lambda_i(f)$ are suitable linear functionals. 
The operator $\mathcal{Q}$ is called a quasi-interpolant if it yields an accurate approximation without solving a global linear system; in local constructions, each $\lambda_i(f)$ depends only on the values of $f$ in a neighborhood of the support of $N_i^p$.

In order to ensure that point-evaluation functionals are well defined, we assume throughout this section that $f \in \mathcal{C}^{-1}([a,b])$, where $[a,b]$ is a bounded interval. 
We consider the spline space $\mathcal{S}^p(T)$ associated with a $(p+1)$-basic knot sequence $T$.
The corresponding basic interval $[t_{p+1}, t_{n+1}]$ is assumed to coincide with $[a,b]$.
\begin{definition}[\cite{lyche_b-splines_2017}]
The quasi-interpolant
\[
\mathcal{Q}f(t) = \sum_{i=0}^n \lambda_i(f)\, N_i^p(t)
\]
is called a \emph{local quasi-interpolant} if the following conditions hold:

(i) Each functional $\lambda_i$ is supported on the interval
\[
I_i := \left[t_{i-v_L},\, t_{i+p+1+v_U}\right] \cap [a,b],
\]
for some integers $v_L, v_U \ge -p$, and $I_i$ has nonempty interior.

(ii) The operator $\mathcal{Q}$ reproduces polynomials of degree up to $l$, i.e.,
\[
\mathcal{Q} g(t) = g(t) \quad \text{for all } g \in \mathbb{P}_l \text{ and all } t \in [a,b],
\]
for some integer $l$ satisfying $0 \le l \le p$.
\end{definition}
\begin{definition}[\cite{lyche_b-splines_2017}]
A local quasi-interpolant $\mathcal{Q}$ is said to be \emph{bounded in the $L_q$-norm}, 
for $1 \le q \le \infty$, if there exists a constant $C_{\mathcal{Q}} > 0$ such that, 
for each functional $\lambda_i$, the estimate
\[
|\lambda_i(f)| \le C_{\mathcal{Q}}\, h_{i,p,T}^{-1/q} \, \|f\|_{L_q(I_i)}
\]
holds for all $f \in C^{-1}(I_i)$.

Here,
$h_{i,p,T}$ is the largest knot interval length contained in 
$I_i \cap \operatorname{supp}(N_i^p)$.
\end{definition}
\begin{proposition}
Let 
\[
\mathcal{Q}f(t) = \sum_{i=0}^n \lambda_i(f)\, N_i^p(t)
\]
be a linear quasi-interpolant. If $\mathcal{Q}$ reproduces $\mathbb{P}_p$ and each functional 
$\lambda_i$ is supported on a single knot interval
\[
[t_{m_i}^+,\, t_{m_i+1}^-] \subset [t_i, t_{i+p+1}], 
\quad \text{with } t_{m_i} < t_{m_i+1},
\]
then $\mathcal{Q}$ reproduces the entire spline space $\mathcal{S}^p(T)$, i.e.,
\[
\mathcal{Q}s(t) = s(t) 
\quad \text{for all } s \in \mathcal{S}^p(T) 
\text{ and } t \in [t_{p+1}, t_{n+1}].
\]
In particular, $\mathcal{Q}$ is a projector onto $\mathcal{S}^p(T)$.
\end{proposition}
\begin{proof}
We refer to \cite{lyche_b-splines_2017} for a detailed proof of this result.
\end{proof}
A general construction of a quasi-interpolant reproducing the polynomial space is given in \cite{lyche_b-splines_2017,lee_examples_2001}. 
For the reader's convenience, we recall it here as follows.

For a fixed index $i$, the coefficient $\lambda_i(f)$ is determined as follows:

(i) Choose an interval 
\[
\hat{I}_i := [t_{m_{L,i}}, t_{m_{U,i}}] \subset [a,b]
\] 
such that 
\[
(t_{m_{L,i}}, t_{m_{U,i}}) \cap (t_i, t_{i+p+1}) \neq \emptyset,
\] 
and such that $m_{U,i} - m_{L,i}$ is uniformly bounded independently of $n$.

(ii) Select a local linear approximation operator $\mathcal{Q}_i$ which can be expressed in terms of B-splines as
\begin{equation}
\label{Poly_quasi}
\mathcal{Q}_i f(t) = \sum_{j=m_{L,i}-p}^{m_{U,i}-1} b_j\, N_j^p(t), 
\quad t \in (t_{m_{L,i}}, t_{m_{U,i}}),
\end{equation}
and which reproduces polynomials of degree up to $l$, i.e.,
\[
\mathcal{Q}_i g(t) = g(t), \quad \forall g \in \mathbb{P}_l, \quad t \in (t_{m_{L,i}}, t_{m_{U,i}}),
\] 
for some fixed $l$ with $0 \le l \le p$.

(iii) Finally, set 
\[
\lambda_i(f) := b_i.
\]

If we replace \eqref{Poly_quasi} by 
\begin{equation} 
\label{Spline_qua}
\mathcal{Q}_i s(t) = s(t), \quad \text{for all } s \in \mathcal{S}^p(T) \text{ and } t \in (t_{m_{L,i}}, t_{m_{U,i}}).
\end{equation}
then the quasi-interpolant $\mathcal{Q}$ reproduces the B-spline space $\mathcal{S}^p(T)$ (see \cite{lyche_b-splines_2017}).

The univariate interpolation and quasi-interpolation operators introduced above can be extended to the multidimensional case via a tensor-product construction.  
Let, for $i=1,\dots,d$, the symbol $\mathcal{Q}_{p_i, T_i}$ denote the univariate operator onto the spline space $\mathcal{S}^{p_i}(T_i)$. We define the multidimensional operator
\[
\mathcal{Q}_{\mathbf{p}, \mathcal{T}}(u) = \left(\mathcal{Q}_{p_1, T_1} \otimes \dots \otimes \mathcal{Q}_{p_d, T_d}\right)(u),
\]
where $\mathbf{p} = (p_1,\dots,p_d)$ and $\mathcal{T} = \{T_1,\dots,T_d\}$.

For a detailed discussion on the tensorization of quasi-interpolants, we refer to \cite{beirao_mathematical_2014}.  
The definition above holds for smooth functions $u$ and is extended by continuity to the appropriate functional spaces (see \cite{schumaker_spline_2007}).

\begin{theorem}[\cite{cottrell_isogeometric_2009}]
\label{interpolation}
Let $r$ and $s$ be integers such that $0 \le r \le s \le \min(p_1, \dots, p_d)+1$. 
Then, there exists a projector(quasi-interpolant) 
\[
\mathcal{Q}_{V_h} : H^r(\Omega) \longrightarrow V_h = \mathcal{S}^{\mathbf{p}}(\mathcal{T})
\] 
satisfying
\begin{equation}
\label{projection}
\| u - \mathcal{Q}_{V_h} u \|_{H^r(\Omega)} \le C h^{s-r} \, \| u \|_{H^s(\Omega)}, 
\quad \forall u \in H^s(\Omega),
\end{equation}
where $C>0$ is a constant independent of $h$.
\end{theorem}

\begin{remark}
The same result holds in the presence of Dirichlet boundary conditions. 
In this case, one can construct a projector 
\[
\mathcal{Q}_{V_{h,\Gamma_D}} : H^r(\Omega) \longrightarrow V_{h, \Gamma_D} 
\] 

which satisfies the estimate \eqref{projection} (see \cite{cottrell_isogeometric_2009} for details).
\end{remark}
We conclude this subsection with the following inverse inequality.

\begin{theorem}[\cite{bazilevs_isogeometric_2006}]
\label{Inv_splines}
Let $k$ and $l$ be integers such that $0 \le k \le l$. Then, for all $v_h \in V_h$, the following estimate holds:
\begin{equation}
\label{inv_ineg}  
| v_h |_{H^l(\Omega)} \le C h^{k-l} \, | v_h |_{H^k(\Omega)},
\end{equation}

where $C>0$ is a constant independent of $h$.
\end{theorem}
We are now in a position to introduce the stabilized discrete formulation of the problem.
\section{Stabilized Discrete Formulation}
The proposed approach relies on a multilevel decomposition of the discrete space into large- and small-scale components. Diffusion-type stabilization terms are added and applied only to the small-scale contributions at each level. 

From a multiscale perspective, these terms account for the influence of unresolved fine-scale structures on the resolved small scales within each hierarchical level.
Consider a domain $\Omega=[a,b]\times[c,d]$ equipped with open knot vectors defining the tensor-product mesh $\mathcal{T}_h$. 
The B-spline space of degree $p$ associated with $\mathcal{T}_h$, introduced above, is denoted by
\[
\mathcal{S}^p(\mathcal{T}_h).
\]

We define the discrete space

\[
V_h := \mathcal{S}^p(\mathcal{T}_h) \cap H_0^1(\Omega).
\]

The standard discrete variational formulation of \eqref{VF_equa} then reads as follows:

\medskip
\noindent
\textbf{Find $u_h \in V_h$ such that}
\begin{equation}
\label{DVF_equa_IGA}
a(u_h,v_h) = (f,v_h),
\qquad \forall v_h \in V_h.
\end{equation}

\medskip
\noindent
\textbf{Multilevel hierarchy.}

Let $\mathcal{M}_h^{(k)}$, $k=1,\dots,L$, be a family of hierarchically coarsened meshes associated with the basic mesh $\mathcal{T}_h$, and set
\[
\mathcal{M}_h^{(0)} = \mathcal{T}_h.
\]
Each level $\mathcal{M}_h^{(k)}$ is obtained from $\mathcal{M}_h^{(k-1)}$ by removing selected interior knots. 
We denote by $H^{(k)}$ the mesh size associated with $\mathcal{M}_h^{(k)}$.

We assume that there exists a constant $C>0$, independent of $h$, such that
\[
H^{(k)} \le C h,
\qquad \forall k=1,\dots,L.
\]
We define the ratio
\[
c_k := \frac{h}{H^{(k)}}, \qquad k = 1,\dots,L.
\]

For each level, we introduce the B-spline space of degree $p$
\[
\mathcal{S}^{p}(\mathcal{M}_h^{(k)}),
\]
and denote by $\Pi_k^{p}$ the associated quasi-interpolation operator onto 
$\mathcal{S}^{p}(\mathcal{M}_h^{(k)})$, satisfying the approximation estimate stated in \eqref{projection}.

\medskip
\noindent
\textbf{Multilevel stabilization.}

For $u_h, v_h \in V_h$, we define the multilevel projection stabilization term as
\[
s_{\mathrm{MQ}}(u_h, v_h)
= \tau_h \sum_{k=1}^{L} c_k \, 
\int_{\Omega}
\left( \boldsymbol{b} \cdot \nabla ( u_h - \Pi_{k}^{p} u_h ) \right)
\left( \boldsymbol{b} \cdot \nabla ( v_h - \Pi_{k}^{p} v_h ) \right) \, dx,
\]
where $\tau_h$ is the stabilization parameter to be defined in the subsequent sections.
Note that, unlike classical LPS methods, the projection $\Pi_k^p$ acts on $u_h$ itself rather than on its streamline derivative. This design leads to a natural multiscale decomposition: the fluctuation $u_h - \Pi_k^p u_h$ captures exactly the fine-scale component that the coarse mesh $\mathcal{M}_h^{(k)}$ cannot resolve. Moreover, since the projection uses degree $p$ (the same as $u_h$), the resulting consistency error benefits from the full approximation order of the B-spline space.

We then define the stabilized bilinear form
\begin{equation}  
\label{weak_MQ}
a_{\mathrm{MQ}}(u,v)
=
a(u,v) + s_{\mathrm{MQ}}(u,v).
\end{equation}

The multilevel projection discretization of \eqref{DVF_equa_IGA} reads:

\medskip
\noindent
\textbf{Find $u_h \in V_h$ such that}
\begin{equation}
\label{VF_MQ}
a_{\mathrm{MQ}}(u_h,v_h)
=
(f,v_h),
\qquad \forall v_h \in V_h.
\end{equation}

\begin{remark}
\label{rem:time_dependent}
Another suitable choice of stabilization operator, which is particularly well-suited for extending the method to time-dependent problems, flows with highly variable advection directions, or problems featuring complex internal layers (e.g., Test 4), is the isotropic penalization of the full gradient of the fine-scale fluctuations:
\[
s_{\mathrm{MQ}}(u_h, v_h) = \tau_h \sum_{k=1}^{L} c_k \int_{\Omega} \nabla(u_h - \Pi_k^p u_h) \cdot \nabla(v_h - \Pi_k^p v_h) \, dx.
\]
This isotropic form provides a balanced stabilization over the entire domain. Furthermore, thanks to the high approximation power of the quasi-interpolants, this full-gradient penalization does not introduce excessive numerical crosswind diffusion.
\end{remark}

\medskip
\noindent
\textbf{Energy norms.} 
We introduce the norm
\[
\| v \|_G
=
\left(
\varepsilon |v|_{1,\Omega}^2
+
\| \sigma^{1/2} v \|_{0,\Omega}^2
\right)^{1/2},
\]
and the stabilized norm
\[
\| v \|_{\mathrm{MQ}}
=
\left(
\| v \|_G^2
+
s_{\mathrm{MQ}}(v,v)
\right)^{1/2}.
\]

For all $v \in H_0^1(\Omega)$, we have
\[
a(v,v) = \| v \|_G^2,
\qquad
a_{\mathrm{MQ}}(v,v) = \| v \|_{\mathrm{MQ}}^2.
\]
Therefore, the multilevel projection discretization \eqref{VF_MQ} admits a unique solution. 

\subsection{Stability and Error Analysis in the MQ Norm}

In this section, we present the general stability and a priori error analysis for the proposed multilevel stabilization. The results in this subsection apply to the general multidimensional setting with variable coefficients.

We define the multilevel approximation constant:
\begin{equation}\label{eq:Sigma}
\Sigma(p,L) := \sum_{k=1}^{L} c_k\, (H^{(k)})^{2p}.
\end{equation}

\begin{theorem}[Coercivity in the MQ norm]\label{thm:coercivity}
For all $v_h \in V_h$:
\[
a_{\mathrm{MQ}}(v_h, v_h) = \|v_h\|_{\mathrm{MQ}}^2.
\]
In particular, the discrete problem \eqref{VF_MQ} admits a unique solution.
\end{theorem}

\begin{proof}
By integration by parts and $v_h|_{\partial\Omega} = 0$:
\[
(\mathbf{b}\cdot\nabla v_h, v_h) = -\tfrac{1}{2}(\operatorname{div}\mathbf{b}\, v_h, v_h).
\]
Therefore:
\begin{align*}
a_{\mathrm{MQ}}(v_h, v_h) 
&= \varepsilon |\nabla v_h|_{L^2(\Omega)}^2 + ((c - \tfrac{1}{2}\operatorname{div}\mathbf{b}) v_h, v_h) + s_{\mathrm{MQ}}(v_h, v_h) \\
&= \varepsilon |v_h|_{H^1(\Omega)}^2 + \|\sigma^{1/2} v_h\|_{L^2(\Omega)}^2 + s_{\mathrm{MQ}}(v_h, v_h) = \|v_h\|_{\mathrm{MQ}}^2. \qedhere
\end{align*}
\end{proof}

\begin{lemma}[Modified Galerkin orthogonality]\label{lem:galerkin_orth}
Let $u \in H^{p+1}(\Omega) \cap H_0^1(\Omega)$ solve the continuous problem \eqref{VF_equa} and $u_h \in V_h$ solve the discrete problem \eqref{VF_MQ}. Then:
\begin{equation}\label{eq:mod_galerkin}
a_{\mathrm{MQ}}(u - u_h,\, v_h) = s_{\mathrm{MQ}}(u,\, v_h) \qquad \forall v_h \in V_h.
\end{equation}
\end{lemma}

\begin{proof}
From the continuous problem: $a(u, v_h) = (f, v_h)$.
From the discrete problem: $a(u_h, v_h) + s_{\mathrm{MQ}}(u_h, v_h) = (f, v_h)$.
Subtracting: $a(u, v_h) - a(u_h, v_h) = s_{\mathrm{MQ}}(u_h, v_h)$.
By bilinearity of $s_{\mathrm{MQ}}$:
\begin{align*}
a_{\mathrm{MQ}}(u - u_h, v_h) &= a(u, v_h) - a(u_h, v_h) + s_{\mathrm{MQ}}(u, v_h) - s_{\mathrm{MQ}}(u_h, v_h) \\
&= (f, v_h) + s_{\mathrm{MQ}}(u, v_h) - [a(u_h, v_h) + s_{\mathrm{MQ}}(u_h, v_h)] \\
&= (f, v_h) + s_{\mathrm{MQ}}(u, v_h) - (f, v_h) \\
&= s_{\mathrm{MQ}}(u, v_h). \qedhere
\end{align*}
\end{proof}

\begin{lemma}[Multilevel consistency error]\label{lem:consistency}
Let $u \in H^{p+1}(\Omega)$. Then, with the constant $C_{\Pi}= \max_k C_{\Pi_k^p}$, we have:
\begin{equation}\label{eq:consistency}
\sup_{v_h \in V_h \setminus \{0\}} \frac{|s_{\mathrm{MQ}}(u,\, v_h)|}{\|v_h\|_{\mathrm{MQ}}} 
\le C_\Pi\, \|\mathbf{b}\|_{L^\infty(\Omega)}\, \tau_h^{1/2}\, \Sigma(p,L)^{1/2}\, \|u\|_{H^{p+1}(\Omega)}.
\end{equation}
\end{lemma}

\begin{proof}
By the Cauchy--Schwarz inequality applied to $s_{\mathrm{MQ}}$:
\[
|s_{\mathrm{MQ}}(u, v_h)| \le s_{\mathrm{MQ}}(u, u)^{1/2}\, s_{\mathrm{MQ}}(v_h, v_h)^{1/2} \le s_{\mathrm{MQ}}(u,u)^{1/2}\, \|v_h\|_{\mathrm{MQ}}.
\]

To bound $s_{\mathrm{MQ}}(u,u)$, we use the approximation property \eqref{projection} with $r=1$, $s=p+1$:
\[
\|\mathbf{b}\cdot\nabla(u - \Pi_k^p u)\|_{L^2(\Omega)} \le \|\mathbf{b}\|_{L^\infty(\Omega)}\, |u - \Pi_k^p u|_{H^1(\Omega)} \le C_\Pi\, \|\mathbf{b}\|_{L^\infty(\Omega)}\, (H^{(k)})^p\, \|u\|_{H^{p+1}(\Omega)}.
\]

Therefore:
\begin{align*}
s_{\mathrm{MQ}}(u,u) 
&= \tau_h \sum_{k=1}^L c_k\, \|\mathbf{b}\cdot\nabla(u - \Pi_k^p u)\|_{L^2(\Omega)}^2 \\
&\le C_\Pi^2\, \|\mathbf{b}\|_{L^\infty(\Omega)}^2\, \tau_h\, \|u\|_{H^{p+1}(\Omega)}^2 \sum_{k=1}^L c_k\, (H^{(k)})^{2p} \\
&= C_\Pi^2\, \|\mathbf{b}\|_{L^\infty(\Omega)}^2\, \tau_h\, \Sigma(p,L)\, \|u\|_{H^{p+1}(\Omega)}^2. \qedhere
\end{align*}
\end{proof}

\begin{theorem}[A priori error estimate in the MQ norm]\label{thm:apriori}
Let $u \in H^{p+1}(\Omega) \cap H_0^1(\Omega)$ be the solution of the continuous problem and $u_h \in V_h$ be the solution of the multilevel projection discretization. Assume the stabilization parameter $\tau_h$ satisfies:
\begin{equation}
0 \le \tau_h \le C_\tau \, c_L^{-1} \, \frac{h^2}{\max\{\varepsilon, \, h \|\mathbf{b}\|_{0,\infty}\}},
\end{equation}
where $C_\tau > 0$ is a positive constant.
Then:
\begin{equation}\label{eq:apriori}
\|u - u_h\|_{\mathrm{MQ}} \le C\, \Phi(\varepsilon, h, \mathbf{b}, \sigma)\, h^p\, \|u\|_{H^{p+1}(\Omega)},
\end{equation}
where
\[
\Phi(\varepsilon, h, \mathbf{b}, \sigma) = \left(\varepsilon + \frac{\|\mathbf{b}\|_{L^\infty(\Omega)}^2}{\sigma_0} + h^2\|\sigma\|_{L^\infty(\Omega)} + \tau_h\, \|\mathbf{b}\|_{L^\infty(\Omega)}^2\, \frac{\Sigma(p,L)}{h^{2p}}\right)^{1/2},
\]
and $C > 0$ depends only on $p$, $L$, and $C_\Pi$.
\end{theorem}

\begin{proof}
Let $w_h = \mathcal{Q}_{V_h} u \in V_h$ be the fine-level quasi-interpolant of $u$, and set $\eta = u - w_h$ and $\xi_h = u_h - w_h \in V_h$.

First, we establish the error equation. From Lemma~\ref{lem:galerkin_orth}, we have $a_{\mathrm{MQ}}(\eta - \xi_h, v_h) = s_{\mathrm{MQ}}(u, v_h)$. Setting $v_h = \xi_h$ yields:
\begin{equation}\label{eq:error_eq}
\|\xi_h\|_{\mathrm{MQ}}^2 = a_{\mathrm{MQ}}(\xi_h, \xi_h) = a_{\mathrm{MQ}}(\eta, \xi_h) - s_{\mathrm{MQ}}(u, \xi_h).
\end{equation}

Next, we bound the standard bilinear form term $a(\eta, \xi_h)$. Applying the Cauchy-Schwarz and Poincaré-Friedrichs inequalities gives:
\begin{equation}\label{eq:bilinear_bound}
|a(\eta, \xi_h)| \le \left(\varepsilon^{1/2}|\eta|_{H^1(\Omega)} + \frac{\|\mathbf{b}\|_{L^\infty(\Omega)}}{\sigma_0^{1/2}}|\eta|_{H^1(\Omega)} + \|\sigma\|_{L^\infty(\Omega)}^{1/2}\|\eta\|_{L^2(\Omega)}\right) \|\xi_h\|_{\mathrm{MQ}}.
\end{equation}

For the stabilization term $s_{\mathrm{MQ}}(\eta, \xi_h)$, we apply the Cauchy--Schwarz inequality to obtain $|s_{\mathrm{MQ}}(\eta, \xi_h)| \le s_{\mathrm{MQ}}(\eta, \eta)^{1/2} \cdot \|\xi_h\|_{\mathrm{MQ}}$. For each level $k$, using the $H^1$-stability of the projection, we have:
\begin{align*}
\|\mathbf{b}\cdot\nabla(\eta - \Pi_k^p \eta)\|_{L^2(\Omega)} 
&\le \|\mathbf{b}\|_{L^\infty(\Omega)}|\eta - \Pi_k^p \eta|_{H^1(\Omega)} \\
&\le \|\mathbf{b}\|_{L^\infty(\Omega)}(1+C_\Pi)|\eta|_{H^1(\Omega)}.
\end{align*}
Therefore, the stabilization term is bounded by:
\begin{equation}\label{eq:sMQ_eta}
s_{\mathrm{MQ}}(\eta, \eta) \le \tau_h\, \|\mathbf{b}\|_{L^\infty(\Omega)}^2\, (1+C_\Pi)^2\, |\eta|_{H^1(\Omega)}^2 \sum_{k=1}^L c_k.
\end{equation}

To bound the consistency error $s_{\mathrm{MQ}}(u, \xi_h)$, we use Lemma~\ref{lem:consistency}:
\[
|s_{\mathrm{MQ}}(u, \xi_h)| \le C_\Pi\, \|\mathbf{b}\|_{L^\infty(\Omega)}\, \tau_h^{1/2}\, \Sigma(p,L)^{1/2}\, \|u\|_{H^{p+1}(\Omega)} \cdot \|\xi_h\|_{\mathrm{MQ}}.
\]

Combining these estimates in the error equation \eqref{eq:error_eq}, and recalling that $a_{\mathrm{MQ}}(\eta, \xi_h) = a(\eta, \xi_h) + s_{\mathrm{MQ}}(\eta, \xi_h)$, we deduce:
\begin{align}
\|\xi_h\|_{\mathrm{MQ}} &\le \varepsilon^{1/2}|\eta|_{H^1(\Omega)} + \frac{\|\mathbf{b}\|_{L^\infty(\Omega)}}{\sigma_0^{1/2}}|\eta|_{H^1(\Omega)} + \|\sigma\|_{L^\infty(\Omega)}^{1/2}\|\eta\|_{L^2(\Omega)} \label{eq:xi_bound} \\
&\quad + s_{\mathrm{MQ}}(\eta,\eta)^{1/2} + C_\Pi\|\mathbf{b}\|_{L^\infty(\Omega)}\tau_h^{1/2}\Sigma(p,L)^{1/2}\|u\|_{H^{p+1}(\Omega)}. \nonumber
\end{align}

By the quasi-interpolant approximation property, the interpolation errors are bounded as follows:
\begin{align}
|\eta|_{H^1(\Omega)} &= |u - \mathcal{Q}_{V_h} u|_{H^1(\Omega)} \le C_\Pi\, h^p\, \|u\|_{H^{p+1}(\Omega)}, \label{eq:interp1} \\
\|\eta\|_{L^2(\Omega)} &= \|u - \mathcal{Q}_{V_h} u\|_{L^2(\Omega)} \le C_\Pi\, h^{p+1}\, \|u\|_{H^{p+1}(\Omega)}. \label{eq:interp0}
\end{align}

Finally, by the triangle inequality, $\|u - u_h\|_{\mathrm{MQ}} \le \|\eta\|_{\mathrm{MQ}} + \|\xi_h\|_{\mathrm{MQ}}$. Substituting the bounds \eqref{eq:interp1}--\eqref{eq:interp0} into the estimates for $\|\eta\|_{\mathrm{MQ}}$ and $\|\xi_h\|_{\mathrm{MQ}}$ yields the final result.
\end{proof}

\begin{remark}[On the choice of stabilization parameter $\tau_h$]
The upper bound on $\tau_h$ in Theorem~\ref{thm:apriori}, inspired by Knobloch~\cite{knobloch_generalization_2010}, is essential for the MQSD inf-sup stability established in Section~\ref{sec:infsup_verification}. In practice, the scaling $\tau_h = c_b\, h$ with a moderate constant $c_b$ naturally satisfies this bound in the advection-dominated regime ($\varepsilon \ll h\|\mathbf{b}\|_{L^\infty}$), since the theoretical condition then reduces to $\tau_h \le C_\tau\, c_L^{-1}\, h / \|\mathbf{b}\|_{L^\infty}$. Throughout the numerical experiments in Section~\ref{sec:numerical}, we adopt $c_b = 0.1$ for one-dimensional problems and $c_b = 0.01$ for two-dimensional problems. These values are consistent with the theoretical bound for all test configurations considered and yield robust stabilization across the full range of Péclet numbers.
\end{remark}

\begin{remark}[On the convergence rate in the MQ norm]\label{rem:rate}
The convergence rate $O(h^p)$ in the natural MQ norm is the standard result for projection-based stabilization methods; the same rate appears in the coercivity-based analysis of LPS ~\cite{knobloch_local_2009, matthies_local_2015, knobloch_generalization_2010}. In LPS, the improved rate $O(h^{p+1/2})$ is obtained by supplementing coercivity with a discrete inf-sup condition that controls the full streamline derivative (see~\cite{knobloch_generalization_2010}, Section~4). The analogous enhanced stability for the present method is established in Section~\ref{sec:infsup_verification}: Theorem~\ref{thm:inf_sup_mqsd} proves a discrete inf-sup condition in the MQSD norm, which provides explicit control of the full streamline derivative $\mathbf{b}\cdot\nabla u_h$. This enhanced stability is the key ingredient that, following the same argument structure as in~\cite{knobloch_generalization_2010} (Section~4), yields the improved $O(h^{p+1/2})$ convergence rate within the multilevel quasi-interpolant framework. However, it should be noted that the rigorous mathematical proof of this enhanced MQSD stability provided in Section~\ref{sec:infsup_verification} is currently restricted to the one-dimensional setting with a constant advection field.
\end{remark}

\section{Enhanced Stability and Streamline Derivative Control}
\label{sec:infsup_verification}

In this section, we establish a discrete inf-sup condition for the bilinear form $a_{\mathrm{MQ}}$ with respect to the enhanced MQSD norm, which provides explicit control of the full streamline derivative $\mathbf{b}\cdot\nabla u_h$. This enhanced stability goes beyond the coercivity-based analysis of Section~4 and is the key ingredient for improved convergence estimates, as discussed in Remark~\ref{rem:rate}.

We present a rigorous proof in the one-dimensional setting with constant advection, where the essential algebraic structure of the multilevel stabilization is most transparent. The robust convergence and effective reduction of oscillations observed in the numerical experiments of Section~\ref{sec:numerical} provide strong evidence that these stability properties extend to variable-coefficient and multidimensional configurations, although a fully general proof remains an open problem.

\subsection{Inf-sup condition hypothesis}
The key to proving enhanced streamline control is an inf-sup condition between the fine space $V_h$ and the coarse space $\mathcal{S}^{p-1}(\mathcal{M}_h^{(L)})$. To this end, we assume a \emph{dyadic coarsening hierarchy}. Specifically, starting from a uniform fine mesh $\mathcal{T}_h$ with mesh size $h$, each progressively coarsened mesh $\mathcal{M}_h^{(k)}$ for $k=1, \dots, L$ is constructed by removing every second interior knot from $\mathcal{M}_h^{(k-1)}$. Consequently, the mesh size at level $k$ is $H^{(k)} = 2^k h$, and each macroscopic element in the coarsest mesh $\mathcal{M}_h^{(L)}$ is the union of exactly $2^L$ fine elements from $\mathcal{T}_h$.

\begin{hypothesis}[\textbf{H1} -- Uniform $L^2$ inf-sup stability]
\label{hyp:inf_sup}
Let $V_h := \mathcal{S}^p(\mathcal{T}_h) \cap H_0^1(I)$ be the B-spline space of degree $p \ge 2$ on the fine mesh $\mathcal{T}_h$, and let $Q_L := \mathcal{S}^{p-1}(\mathcal{M}_h^{(L)})$ be the B-spline space of degree $p-1$ on the hierarchically coarsened mesh $\mathcal{M}_h^{(L)}$. Then there exists a constant $\alpha_{\mathrm{MQ}} > 0$, depending only on $p$ and the coarsening structure, but \emph{independent of $h$}, such that
\begin{equation}
\label{inf_sup1}
\sup_{v_h \in V_h \setminus \{0\}}
\frac{(v_h,\, q_h)_{0,I}}{\|v_h\|_{0,I}}
\;\ge\;
\alpha_{\mathrm{MQ}} \, \|q_h\|_{0,I}
\qquad
\forall\, q_h \in Q_L.
\end{equation}
\end{hypothesis}

\subsection{Empirical Validation via Dimension Condition}
To validate Hypothesis~\ref{hyp:inf_sup} for the dyadic coarsening strategy defined above, we note that the positivity of the inf-sup constant $\alpha_{\mathrm{MQ}}(h) > 0$ for each fixed $h$ can be guaranteed by invoking the Bressan--Sangalli dimension condition~\cite{bressan_isogeometric_2013}. The uniform $h$-independence of this constant is supported by the numerical evidence in Table~\ref{tab:infsup}. 

\begin{theorem}[Bressan--Sangalli \cite{bressan_isogeometric_2013}]
\label{thm:BS31}
Let $S_1$ and $S_2$ be two univariate spline spaces defined on the same domain $I$, with respective B-spline bases $\mathcal{B}_1$ and $\mathcal{B}_2$. If for every continuous subinterval $J \subset I$, the dimension condition
\begin{equation}
\label{eq:dim_cond}
\dim \{ B \in \mathcal{B}_1 \mid \operatorname{supp}(B) \cap J \neq \emptyset \} 
\;\ge\; 
\dim \{ C \in \mathcal{B}_2 \mid \operatorname{supp}(C) \subset J \}
\end{equation}
holds, then the $L^2$ coupling matrix between $S_1$ and $S_2$ has full row rank, meaning $S_2 \cap S_1^\perp = \{0\}$.
\end{theorem}

\begin{theorem}[Positivity of the discrete inf-sup constant]
\label{thm:positivity}
Assume the hierarchically coarsened mesh $\mathcal{M}_h^{(L)}$ is constructed via dyadic coarsening of $\mathcal{T}_h$. For each $h > 0$, there exists $\alpha_{\mathrm{MQ}}(h) > 0$ such that
\begin{equation}
\inf_{q_h \in Q_L \setminus \{0\}} \sup_{v_h \in V_h \setminus \{0\}}
\frac{(v_h, q_h)_{0,I}}{\|v_h\|_{0,I} \|q_h\|_{0,I}} \ge \alpha_{\mathrm{MQ}}(h).
\end{equation}
\end{theorem}

\begin{proof}
Because the coarse spaces are nested ($Q_L \subset Q_1$ for $L \ge 1$), it suffices to verify the dimension condition for the basic coarsening level $L=1$. The Bressan--Sangalli dimension condition requires that for any continuous subinterval $J \subset I$, the number of fine basis functions intersecting $J$ is greater than or equal to the number of coarse basis functions completely supported in $J$. This algebraic bound holds globally for dyadic meshes. 

Consequently, the $L^2$ scalar product between $V_h$ and $Q_L$ has full row rank, meaning $Q_L \cap V_h^\perp = \{0\}$.
Let $P_{V_h}: L^2(I) \to V_h$ denote the $L^2$-orthogonal projection onto the fine space. Since $Q_L \cap V_h^\perp = \{0\}$, the mapping $q_h \mapsto \|P_{V_h} q_h\|_{0,I}$ defines a valid norm on the finite-dimensional space $Q_L$. By the equivalence of norms in finite-dimensional spaces, there exists a strictly positive constant $\alpha_{\mathrm{MQ}}(h)$ such that 
\[
\|P_{V_h} q_h\|_{0,I} \ge \alpha_{\mathrm{MQ}}(h) \|q_h\|_{0,I} \qquad \forall q_h \in Q_L.
\]
By the standard property of the $L^2$ orthogonal projection, we have $\sup_{v_h \in V_h \setminus \{0\}} \frac{(v_h, q_h)_{0,I}}{\|v_h\|_{0,I}} = \|P_{V_h} q_h\|_{0,I}$, which concludes the proof.
\end{proof}

To verify the asymptotic $h$-independence of the constant, the exact discrete inf-sup constant, denoted by $\hat{\beta}(h) = \alpha_{\mathrm{MQ}}(h)$, was computed numerically. The computations were executed across varying polynomial degrees $p \in \{2, 3\}$ and local refinement levels $L \in \{1, 2, 3\}$. The results are summarized in Table~\ref{tab:infsup}.

\begin{table}[H]
\centering
\caption{Computed inf-sup constant $\hat{\beta}(h)$ for varying mesh refinement (degree $p$, level $L$, $n$ fine elements) using the dyadic coarsening strategy.}
\label{tab:infsup}
\begin{tabular}{ccccccccc}
\hline
$p$ & $L$ & $n=8$ & $n=16$ & $n=32$ & $n=64$ & $n=128$ & $n=256$ & $n=512$ \\
\hline
2 & 1 & 0.8641 & 0.8657 & 0.8657 & 0.8657 & 0.8657 & 0.8657 & 0.8657 \\
2 & 2 & 0.9218 & 0.9319 & 0.9327 & 0.9327 & 0.9327 & 0.9327 & 0.9327 \\
2 & 3 & 0.9422 & 0.9617 & 0.9666 & 0.9670 & 0.9670 & 0.9670 & 0.9670 \\
3 & 1 & 0.8246 & 0.8296 & 0.8297 & 0.8297 & 0.8297 & 0.8297 & 0.8297 \\
3 & 2 & 0.8992 & 0.9143 & 0.9166 & 0.9167 & 0.9167 & 0.9167 & 0.9167 \\
3 & 3 & 0.9255 & 0.9509 & 0.9580 & 0.9591 & 0.9592 & 0.9592 & 0.9592 \\
\hline
\end{tabular}
\end{table}

A notable feature of Table~\ref{tab:infsup} is that the computed inf-sup constants converge rapidly to a fixed value as $n$ increases, confirming the $h$-independence of $\alpha_{\mathrm{MQ}}$. Moreover, the constant improves significantly with the number of levels $L$: for $L=3$ and $p=3$, the asymptotic value reaches $\alpha_{\mathrm{MQ}} \approx 0.96$, exhibiting negligible $h$-dependence from $n=32$ onward. This suggests that the multilevel hierarchy inherently reinforces the inf-sup stability, a property that has no counterpart in classical two-level methods.

\subsection{Inf-Sup Stability in the MQSD Norm}

We show that under Hypothesis~\ref{hyp:inf_sup}, the method is stable with respect to the enhanced norm
\[
\| v \|_{\mathrm{MQSD}}
=
\left(
\| v \|_G^2
+
s_{\mathrm{MQ}}(v,v)
+
\tau_h c_L \| \mathbf{b} \cdot \nabla v \|_{0,I}^2
\right)^{1/2}.
\]
The abbreviation \mbox{SD} refers to the additional streamline derivative contribution.

\begin{theorem}[MQSD Stability]
\label{thm:inf_sup_mqsd}
Assume $\mathbf{b}$ is constant. Let the stabilization parameter $\tau_h$ satisfy
\begin{equation}
0 \le \tau_h \le C_\tau \, c_L^{-1} \, \frac{h^2}{\max\{\varepsilon, \, h \|\mathbf{b}\|_{0,\infty}\}},
\label{eq:stabilization_parameter}
\end{equation}
where $C_\tau > 0$ is a positive constant.
Then the bilinear form $a_{\mathrm{MQ}}$ satisfies the discrete inf-sup condition
\begin{equation}
\label{inf_sup2}
\sup_{v_h \in V_h} \frac{a_{\mathrm{MQ}}(u_h, v_h)}{\|v_h\|_{\mathrm{MQSD}}} \ge \beta \, \|u_h\|_{\mathrm{MQSD}} 
\quad \forall u_h \in V_h,
\end{equation}
where $\beta > 0$ is a constant independent of $h$ and $\varepsilon$.
\end{theorem}
 
\begin{proof}
Let $u_h \in V_h$. We construct $v_h \in V_h$ such that
\[
a_{\mathrm{MQ}}(u_h, v_h) \ge C_1 \| u_h \|_{\mathrm{MQSD}}^2 \quad \text{and} \quad \| v_h \|_{\mathrm{MQSD}} \le C_2 \| u_h \|_{\mathrm{MQSD}},
\]
which implies the desired inf-sup condition.

Since $\mathbf{b}$ is constant and $\Pi_L^p u_h \in \mathcal{S}^p(\mathcal{M}_h^{(L)})$, its derivative $\mathbf{b} \cdot \nabla (\Pi_L^p u_h)$ is a spline of degree $p-1$ on the coarse mesh $\mathcal{M}_h^{(L)}$. Thus, it follows exactly that $\mathbf{b} \cdot \nabla (\Pi_L^p u_h) \in Q_L$.

By Hypothesis~\ref{hyp:inf_sup} (\textbf{H1}), we define $z_h \in V_h$ such that
\begin{align}
(z_h, u_H) &= c_L \tau_h (\mathbf{b} \cdot \nabla (\Pi_L^p u_h), u_H) \quad \forall u_H \in Q_L, \label{eq:zh1} \\
\| z_h \|_0 &\le c_L \alpha_{\mathrm{MQ}}^{-1} \tau_h \| \mathbf{b} \cdot \nabla (\Pi_L^p u_h) \|_0. \label{eq:zh2}
\end{align}

Testing \eqref{eq:zh1} with $u_H = \mathbf{b} \cdot \nabla (\Pi_L^p u_h) \in Q_L$, we obtain:
\[
(z_h, \mathbf{b} \cdot \nabla (\Pi_L^p u_h)) = c_L \tau_h \| \mathbf{b} \cdot \nabla (\Pi_L^p u_h) \|_0^2.
\]

We split the advective term evaluated at $(u_h, z_h)$ using the multiscale decomposition:
\begin{align*}
(\mathbf{b} \cdot \nabla u_h, z_h) &= (z_h, \mathbf{b} \cdot \nabla (\Pi_L^p u_h)) + (z_h, \mathbf{b} \cdot \nabla (u_h - \Pi_L^p u_h)) \\
&= c_L \tau_h \| \mathbf{b} \cdot \nabla (\Pi_L^p u_h) \|_0^2 + (z_h, \mathbf{b} \cdot \nabla (u_h - \Pi_L^p u_h)).
\end{align*}

Hence, the MQ-stabilized bilinear form can be written as:
\[
a_{\mathrm{MQ}}(u_h, z_h) = c_L \tau_h \| \mathbf{b} \cdot \nabla (\Pi_L^p u_h) \|_0^2 + (a) + (b) + (c) + s_{\mathrm{MQ}}(u_h, z_h),
\]
where
\begin{align*}
(a) &= (z_h, \mathbf{b} \cdot \nabla (u_h - \Pi_L^p u_h)), \\
(b) &= \varepsilon (\nabla u_h, \nabla z_h), \\
(c) &= (c u_h, z_h).
\end{align*}

Using \eqref{eq:zh2}, inverse inequalities, and the upper bound on $\tau_h$, we obtain:
\begin{equation}
\varepsilon^{1/2} | z_h |_1 \le C_\tau^{1/2} C_{inv} \alpha_{\mathrm{MQ}}^{-1} (c_L \tau_h)^{1/2} \| \mathbf{b} \cdot \nabla (\Pi_L^p u_h) \|_0. \label{eq:zh_bound3}
\end{equation}

Now we bound the terms $(a)$, $(b)$, and $(c)$ using the Cauchy-Schwarz inequality.

For $(a)$ and $(b)$:
\begin{align*}
|(a)| &\le \| z_h \|_0 \| \mathbf{b} \cdot \nabla (u_h - \Pi_L^p u_h) \|_0 \le \alpha_{\mathrm{MQ}}^{-1} (c_L \tau_h)^{1/2} \| \mathbf{b} \cdot \nabla (\Pi_L^p u_h) \|_0 \left( c_L \tau_h \| \mathbf{b} \cdot \nabla (u_h - \Pi_L^p u_h) \|_0^2 \right)^{1/2}, \\
|(b)| &\le \varepsilon | u_h |_1 | z_h |_1 \le C_\tau^{1/2} C_{inv} \alpha_{\mathrm{MQ}}^{-1} (c_L \tau_h)^{1/2} \| \mathbf{b} \cdot \nabla (\Pi_L^p u_h) \|_0 \left( \varepsilon | u_h |_1^2 \right)^{1/2}.
\end{align*}
Summing them and applying the discrete Cauchy-Schwarz inequality for sums ($x_1 y_1 + x_2 y_2 \le \sqrt{x_1^2 + x_2^2} \sqrt{y_1^2 + y_2^2}$):
\begin{align*}
|(a) + (b)| &\le \sqrt{2} \left( \alpha_{\mathrm{MQ}}^{-1} + C_\tau^{1/2} C_{inv} \alpha_{\mathrm{MQ}}^{-1} \right) (c_L \tau_h)^{1/2} \| \mathbf{b} \cdot \nabla (\Pi_L^p u_h) \|_0 \left[ \varepsilon |u_h|_1^2 + c_L \tau_h \| \mathbf{b} \cdot \nabla (u_h - \Pi_L^p u_h) \|_0^2 \right]^{1/2} \\
&\le \sqrt{2} \left( \alpha_{\mathrm{MQ}}^{-1} + C_\tau^{1/2} C_{inv} \alpha_{\mathrm{MQ}}^{-1} \right) (c_L \tau_h)^{1/2} \| \mathbf{b} \cdot \nabla (\Pi_L^p u_h) \|_0 \| u_h \|_{\mathrm{MQ}}.
\end{align*}
We define $C_{ab} = \sqrt{2} \left( \alpha_{\mathrm{MQ}}^{-1} + C_\tau^{1/2} C_{inv} \alpha_{\mathrm{MQ}}^{-1} \right)$.

For $(c)$, assuming $c \in L^\infty(\Omega)$, we use the bounds $\| \Pi_L^p u_h \|_0 \le C_{\Pi_L^p} \| u_h \|_0$ (where $C_{\Pi_L^p}$ is the stability constant for the projector $\Pi_L^p$), $\| u_h \|_0 \le \sigma_0^{-1/2} \| \sigma^{1/2} u_h \|_0$, and the condition $\tau_h \le c_L^{-1} C_\tau \frac{h}{\| \mathbf{b} \|_{0,\infty}}$:
\begin{align*}
|(c)| &\le \| c \|_{0,\infty} \| u_h \|_0 \| z_h \|_0 \\
&\le \frac{\| c \|_{0,\infty}}{\sigma_0^{1/2}} \| \sigma^{1/2} u_h \|_0 \left( c_L \tau_h \alpha_{\mathrm{MQ}}^{-1} \| \mathbf{b} \cdot \nabla (\Pi_L^p u_h) \|_0 \right) \\
&\le \frac{\| c \|_{0,\infty}}{\sigma_0^{1/2}} \alpha_{\mathrm{MQ}}^{-1} (c_L \tau_h) \| \mathbf{b} \|_{0,\infty} \| \nabla (\Pi_L^p u_h) \|_0 \| \sigma^{1/2} u_h \|_0 \\
&\le \frac{\| c \|_{0,\infty}}{\sigma_0^{1/2}} \alpha_{\mathrm{MQ}}^{-1} (c_L \tau_h) \| \mathbf{b} \|_{0,\infty} \left( C_{inv} h^{-1} C_{\Pi_L^p} \| u_h \|_0 \right) \| \sigma^{1/2} u_h \|_0 \\
&\le \frac{\| c \|_{0,\infty}}{\sigma_0} \alpha_{\mathrm{MQ}}^{-1} (c_L \tau_h) \| \mathbf{b} \|_{0,\infty} C_{inv} C_{\Pi_L^p} h^{-1} \| \sigma^{1/2} u_h \|_0^2 \\
&\le \frac{\| c \|_{0,\infty}}{\sigma_0} \alpha_{\mathrm{MQ}}^{-1} C_{inv} C_{\Pi_L^p} C_\tau \| \sigma^{1/2} u_h \|_0^2.
\end{align*}
This implies $|(c)| \le C_c \| u_h \|_{\mathrm{MQ}}^2$, where $C_c = \frac{\| c \|_{0,\infty}}{\sigma_0} \alpha_{\mathrm{MQ}}^{-1} C_{inv} C_{\Pi_L^p} C_\tau$.

Now let us estimate the term $s_{\mathrm{MQ}}(u_h, z_h)$:
$$
|s_{\mathrm{MQ}}\left(u_h, z_h\right)| \leq \sqrt{s_{\mathrm{MQ}}\left(u_h, u_h\right)} \sqrt{s_{\mathrm{MQ}}\left(z_h, z_h\right)} \le \| u_h \|_{\mathrm{MQ}} \sqrt{s_{\mathrm{MQ}}(z_h, z_h)}.
$$
For all $k \in \{1\dots L\}$, we have
\begin{align*}
\| \mathbf{b} \cdot \nabla (z_h - \Pi_k^p(z_h)) \|_0 &\le \|\mathbf{b} \|_{0,\infty} C_{inv} h^{-1} (1+C_{\Pi_k^p}) \| z_h \|_0 \\
&\le c_L \tau_h \|\mathbf{b} \|_{0,\infty} C_{inv} h^{-1} (1+C_{\Pi_k^p}) \alpha_{\mathrm{MQ}}^{-1} \| \mathbf{b} \cdot \nabla (\Pi_L^p u_h) \|_0.
\end{align*}
Using the condition $\tau_h \le C_\tau c_L^{-1} \frac{h}{\|\mathbf{b}\|_{0,\infty}}$, we get:
\begin{equation}
\| \mathbf{b} \cdot \nabla (z_h - \Pi_k^p(z_h)) \|_0 \le C_{\tau} C_{inv} (1+C_{\Pi_k^p}) \alpha_{\mathrm{MQ}}^{-1} \| \mathbf{b} \cdot \nabla (\Pi_L^p u_h) \|_0.
\end{equation}
Then:
\begin{equation}
s_{\mathrm{MQ}}(z_h, z_h) \le C_{s_{\mathrm{MQ}}}^2 c_L \tau_h \| \mathbf{b} \cdot \nabla (\Pi_L^p u_h) \|_0^2,
\end{equation}
where $C_{s_{\mathrm{MQ}}}^2 = C_{\tau}^2 C_{inv}^2 \alpha_{\mathrm{MQ}}^{-2} c_L^{-1} \sum_{k=1}^{L} c_k (1+C_{\Pi_k^p})^2$.
This yields:
\begin{equation}
|s_{\mathrm{MQ}}(u_h, z_h)| \le C_{s_{\mathrm{MQ}}} (c_L \tau_h)^{1/2} \| \mathbf{b} \cdot \nabla (\Pi_L^p u_h) \|_0 \| u_h \|_{\mathrm{MQ}}.
\end{equation}

Combining the bounds, we have:
\begin{align*}
a_{\mathrm{MQ}}(u_h, z_h) &\ge c_L \tau_h \| \mathbf{b} \cdot \nabla (\Pi_L^p u_h) \|_0^2 - |(a) + (b)| - |(c)| - |s_{\mathrm{MQ}}(u_h, z_h)| \\
&\ge c_L \tau_h \| \mathbf{b} \cdot \nabla (\Pi_L^p u_h) \|_0^2 - (C_{ab} + C_{s_{\mathrm{MQ}}}) (c_L \tau_h)^{1/2} \| \mathbf{b} \cdot \nabla (\Pi_L^p u_h) \|_0 \| u_h \|_{\mathrm{MQ}} - C_c \| u_h \|_{\mathrm{MQ}}^2.
\end{align*}
Using Young's inequality ($xy \le \frac{1}{2}x^2 + \frac{1}{2}y^2$) on the cross term, we subtract exactly $\frac{1}{2}$ of the stabilization term:
\begin{align*}
(C_{ab} + C_{s_{\mathrm{MQ}}}) (c_L \tau_h)^{1/2} \| \mathbf{b} \cdot \nabla (\Pi_L^p u_h) \|_0 \| u_h \|_{\mathrm{MQ}} \le \frac{1}{2} c_L \tau_h \| \mathbf{b} \cdot \nabla (\Pi_L^p u_h) \|_0^2 + \frac{1}{2} (C_{ab} + C_{s_{\mathrm{MQ}}})^2 \| u_h \|_{\mathrm{MQ}}^2.
\end{align*}
Thus,
\begin{align*}
a_{\mathrm{MQ}}(u_h, z_h) \ge \frac{1}{2} c_L \tau_h \| \mathbf{b} \cdot \nabla (\Pi_L^p u_h) \|_0^2 - C_3 \| u_h \|_{\mathrm{MQ}}^2,
\end{align*}
where $C_3 = \frac{1}{2} (C_{ab} + C_{s_{\mathrm{MQ}}})^2 + C_c$.

To recover the full streamline derivative, we apply the triangle inequality:
\[
c_L \tau_h \| \mathbf{b} \cdot \nabla u_h \|_0^2 \le 2 c_L \tau_h \| \mathbf{b} \cdot \nabla (\Pi_L^p u_h) \|_0^2 + 2 c_L \tau_h \| \mathbf{b} \cdot \nabla (u_h - \Pi_L^p u_h) \|_0^2.
\]
This implies:
\[
\frac{1}{2} c_L \tau_h \| \mathbf{b} \cdot \nabla (\Pi_L^p u_h) \|_0^2 \ge \frac{1}{4} c_L \tau_h \| \mathbf{b} \cdot \nabla u_h \|_0^2 - \frac{1}{2} s_{\mathrm{MQ}}(u_h, u_h).
\]

Substituting this lower bound back into the bilinear form yields:
\[
a_{\mathrm{MQ}}(u_h, z_h) \ge \frac{1}{4} c_L \tau_h \| \mathbf{b} \cdot \nabla u_h \|_0^2 - C_4 \| u_h \|_{\mathrm{MQ}}^2,
\]
where $C_4 = C_3 + \frac{1}{2}$.

Finally, define $v_h := 4 z_h + (1 + 4 C_4) u_h \in V_h$. We evaluate the bilinear form:
\begin{align*}
a_{\mathrm{MQ}}(u_h, v_h) &= 4 a_{\mathrm{MQ}}(u_h, z_h) + (1 + 4 C_4) a_{\mathrm{MQ}}(u_h, u_h) \\
&\ge 4 \left( \frac{1}{4} c_L \tau_h \| \mathbf{b} \cdot \nabla u_h \|_0^2 - C_4 \| u_h \|_{\mathrm{MQ}}^2 \right) + (1 + 4 C_4) \| u_h \|_{\mathrm{MQ}}^2 \\
&= c_L \tau_h \| \mathbf{b} \cdot \nabla u_h \|_0^2 - 4 C_4 \| u_h \|_{\mathrm{MQ}}^2 + \| u_h \|_{\mathrm{MQ}}^2 + 4 C_4 \| u_h \|_{\mathrm{MQ}}^2 \\
&= \| u_h \|_{\mathrm{MQ}}^2 + c_L \tau_h \| \mathbf{b} \cdot \nabla u_h \|_0^2 \\
&= \| u_h \|_{\mathrm{MQSD}}^2.
\end{align*}
Furthermore, since $z_h$ is bounded continuously by $u_h$, we have $\| v_h \|_{\mathrm{MQSD}} \le 4 \| z_h \|_{\mathrm{MQSD}} + (1 + 4 C_4) \| u_h \|_{\mathrm{MQSD}} \le C_V \| u_h \|_{\mathrm{MQSD}}$. Thus, the discrete inf-sup condition holds with $\beta = \frac{1}{C_V}$.
This completes the proof.
\end{proof}

\begin{remark}[Conditionality of MQSD Stability]
We emphasize that the enhanced stability established in Theorem~\ref{thm:inf_sup_mqsd} is conditional on Hypothesis~\ref{hyp:inf_sup}, which guarantees the $h$-independent inf-sup stability of the $L^2$ projection between the fine and coarse spaces.
\end{remark}
\section{Numerical Tests}
\label{sec:numerical}

While the theoretical framework and error analysis in the preceding sections were rigorously developed for the general advection-diffusion-reaction equation, a key strength of the proposed multilevel quasi-interpolant stabilization is its robustness across different flow regimes. To demonstrate this versatility, the following numerical experiments investigate not only the full equation but also singular limit cases where the diffusion vanishes ($\varepsilon = 0$). For instance, we explicitly include tests on a 1D advection-reaction problem (Section~\ref{sec:1d-advection-reaction}) and a 2D pure advection problem (Section~\ref{sec:2d-pure-advection}) to highlight the method's ability to seamlessly handle these challenging degenerate scenarios without requiring any structural modifications. Throughout all numerical experiments, we adopt the dyadic coarsening hierarchy introduced in Section~5.1.

Furthermore, it is important to note that the theoretical quasi-interpolant satisfying the optimal approximation estimates of Theorem~\ref{interpolation} is generally constructed using dual basis functions (see, e.g., \cite{cottrell_isogeometric_2009}), which often lack a simple, explicitly computable form. Fortunately, highly practical quasi-interpolants can be constructed following the framework proposed by Lyche et al.~\cite{lyche_local_1975, lyche_b-splines_2017}, as detailed in Section~\ref{sec:quasi_interpolant}. While this framework allows for various choices of local projection operators, in all our numerical experiments we specifically employ the Greville quasi-interpolant. For a given function $u$ and B-spline degree $p$, the Greville quasi-interpolant $\Pi_G^p$ is defined mathematically as:
\begin{equation}
\label{eq:greville_definition}
\Pi_G^p u(x) = \sum_{i} u(\bar{t}_i) N_i^p(x),
\end{equation}
where $N_i^p(x)$ are the B-spline basis functions and $\bar{t}_i$ are the Greville abscissae, defined as the knot averages:
\begin{equation}
\bar{t}_i = \frac{t_{i+1} + t_{i+2} + \dots + t_{i+p}}{p}.
\end{equation}
The Greville quasi-interpolant strictly follows the Lyche construction paradigm. We adopt it exclusively in our numerical implementation because it requires only simple point evaluations of the target function, making it a highly practical and straightforward choice for implementation while still providing robust multilevel stabilization. While the Greville quasi-interpolant does not possess proven optimal approximation estimates in the sense of Theorem~\ref{interpolation}, it shares the fundamental structural properties of the Lyche framework (locality, polynomial reproduction, projector property), and the numerical results below confirm that the theoretical stabilization mechanism remains fully effective with this practical choice.
\subsection{Test 1: One-dimensional advection-reaction problem}\label{sec:1d-advection-reaction}

We consider the one-dimensional advection-reaction problem:
\begin{equation} \label{eq:test1_adv_rea}
\begin{cases}
u'(x) + u(x) = f(x), & x \in \Omega = (0,1), \\[2mm]
u(0) = 0, &
\end{cases}
\end{equation}
where the forcing term is defined as
\[
f(x) =
\begin{cases}
5, & \text{for} \frac{1}{3} \leq x \leq \frac{2}{3}, \\
0, & \text{otherwise}.
\end{cases}
\]

This example allows us to investigate the performance of the Multilevel stabilized method for a problem with discontinuous source terms, where the exact solution is only piecewise smooth. For this test, we set the multilevel parameter to $L=4$.
\begin{figure}[H]
    \centering
    \begin{subfigure}[b]{0.6\textwidth}
        \centering
        \includegraphics[width=\textwidth]{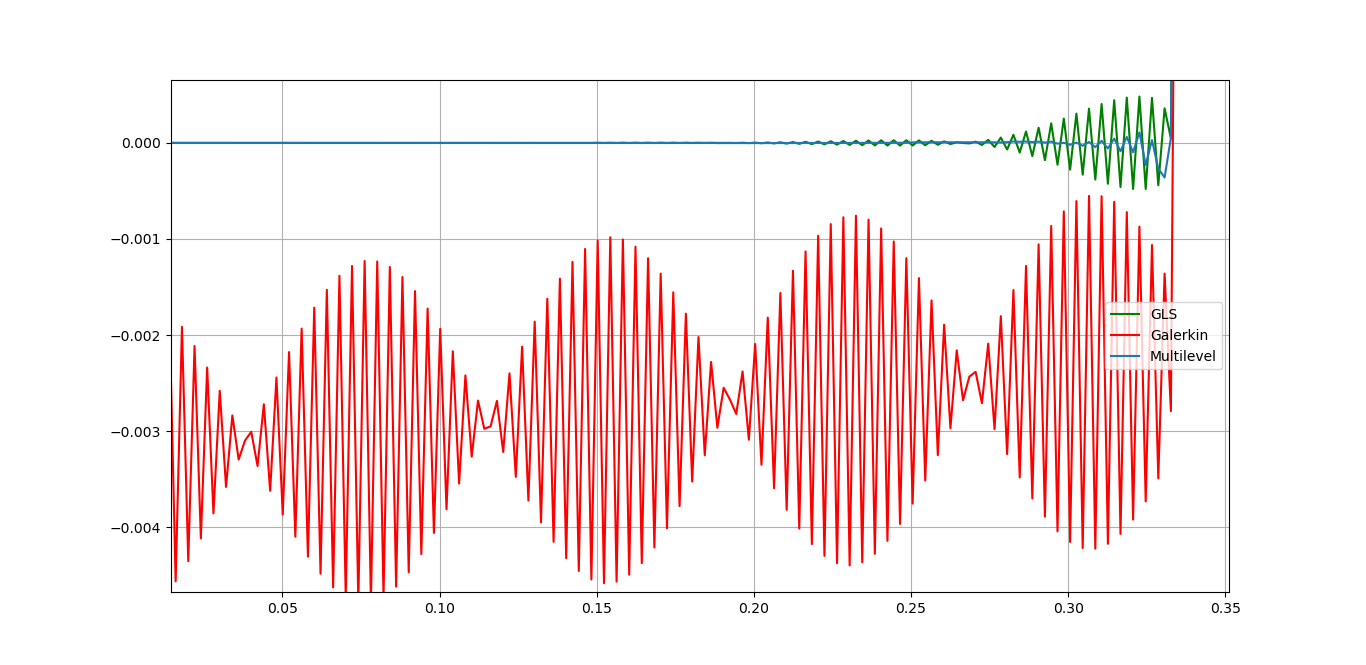} 
        \caption{Zoom on the interval $[0,1/3]$ showing the Multilevel solution compared with Galerkin and GLS solutions for $p=3$ and $n_e=512$.}
        \label{fig:zoom_adv_rea}
    \end{subfigure}
    \hfill
    \begin{subfigure}[b]{0.35\textwidth}
        \centering
        \includegraphics[width=\textwidth]{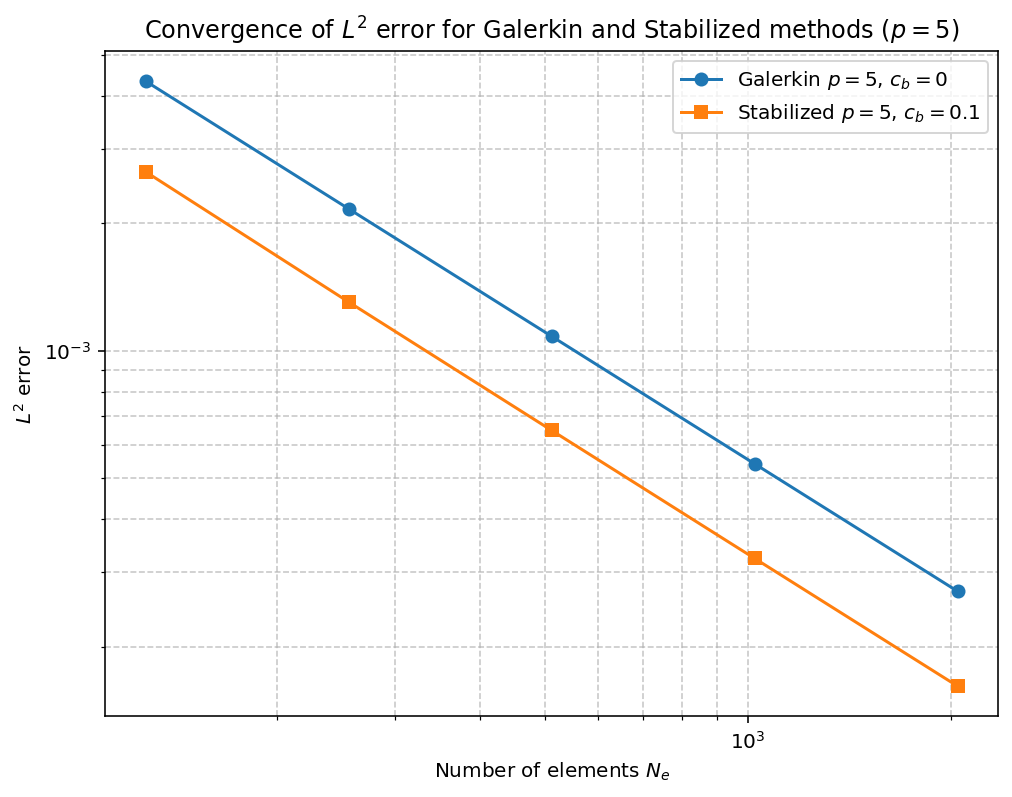} 
        \caption{Convergence history for polynomial order $p=5$.}
        \label{fig:conv_plot}
    \end{subfigure}
    
    \caption{Results for the 1D advection-reaction problem \eqref{eq:test1_adv_rea} (Test 1). (a) Local zoom comparing the Multilevel, Galerkin, and GLS solutions for $p=3$ and $n_e=512$. (b) $L^2$-error convergence history for polynomial degree $p=5$.}
    \label{fig:adv_rea_results}
\end{figure}
\subsection{Test 2: One-dimensional advection-diffusion problem} \label{1D_ADV}

We note that Tests 2, 5, and 6 explore degenerate regimes that include $\sigma_0 = 0$, extending beyond the strict coercivity assumption ($\sigma_0 > 0$) of the theoretical analysis.

We consider the one-dimensional linear advection-diffusion problem
\begin{equation}\label{pd0}
\begin{cases}
-\varepsilon u'' + \beta u' = 1, & x \in (0,1), \\[1mm]
u(0) = u(1) = 0.
\end{cases}
\end{equation}
For this numerical experiment, the multilevel parameter is set to $L=5$.

It is well known that the solution \(u\) of \eqref{pd0} can develop boundary layers when \(0 < \varepsilon \ll |\beta|\), i.e., in the advection-dominated regime. To evaluate the performance of our stabilization strategy under extreme conditions, we set \(\varepsilon = 10^{-5}\) and \(\beta = 1\).

Figure~\ref{Advec_Galerkin} shows the numerical solution obtained using the standard Galerkin discretization, while Figure~\ref{fig:adv_diff_results} depicts the solution computed with our stabilized method. It is clear that the Galerkin solution exhibits spurious oscillations throughout the domain, which are effectively suppressed in the stabilized solution. Minor oscillations may remain near the boundary layers due to steep solution gradients, but the overall accuracy and stability are significantly improved.
\begin{figure}[H]
\centering
\includegraphics[width=8cm]{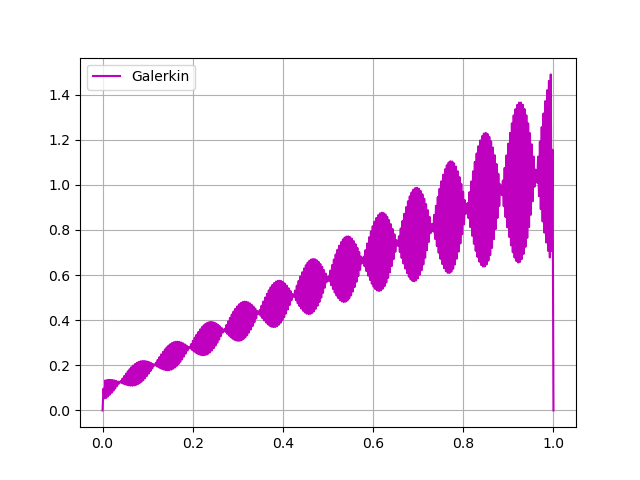}
\caption{Standard Galerkin solution for the 1D advection-diffusion problem \eqref{pd0} (Test 2) with $\varepsilon=10^{-5}$, $p=5$, and $n_e=512$.}
\label{Advec_Galerkin}
\end{figure}

\begin{figure}[H]
    \centering
    \begin{subfigure}[b]{0.45\textwidth}
        \centering
        \includegraphics[width=\textwidth]{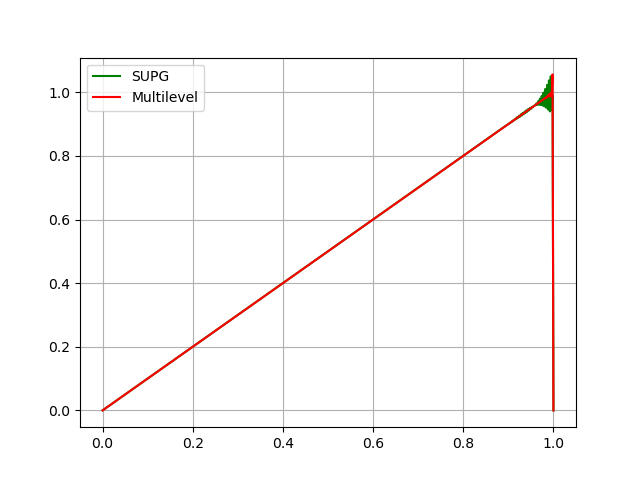} 
        \caption{Numerical solutions using SUPG and Multilevel stabilization.}
        \label{fig:adv_diff_full}
    \end{subfigure}
    \hfill
    \begin{subfigure}[b]{0.45\textwidth}
        \centering
        \includegraphics[width=\textwidth]{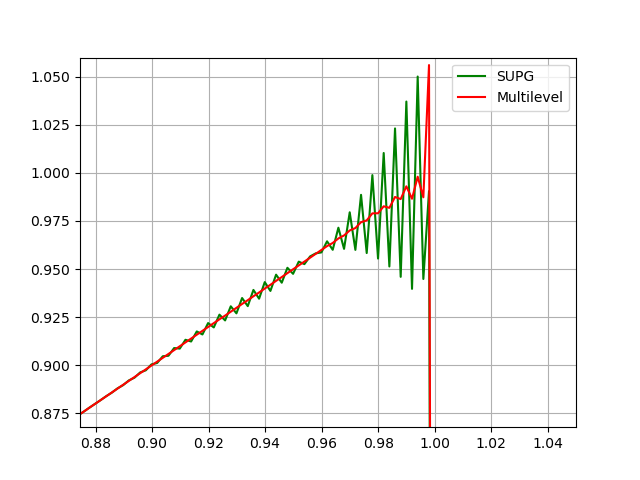} 
        \caption{Zoom near the boundary layer.}
        \label{fig:adv_diff_zoom}
    \end{subfigure}

    \caption{Comparison of stabilized solutions for the 1D advection-diffusion problem \eqref{pd0} (Test 2) using $p=5$ and $n_e=512$. (a) Numerical solutions over the full domain using SUPG and Multilevel stabilization. (b) Detailed zoom near the boundary layer ($x=1$).}
    \label{fig:adv_diff_results}
\end{figure}

\begin{figure}[H]
    \centering
    \includegraphics[width=0.45\textwidth]{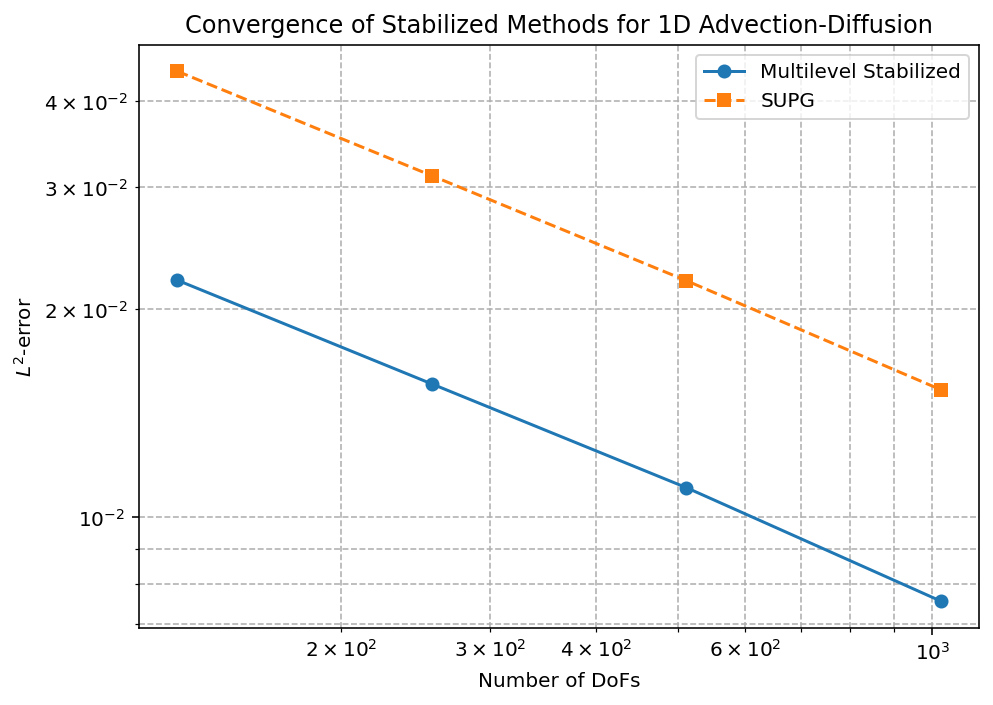}
    \caption{Relative $L^2$-error convergence plot for the 1D advection-diffusion problem \eqref{pd0} (Test 2) with polynomial degree $p=5$.}
    \label{Comp_multilevel}
\end{figure}

\begin{remark}

\label{remark:multilevel_oscillations}
As the number of levels increases, spurious oscillations are progressively reduced. 
This behavior is illustrated in Figure~\ref{fig:boundary_layer_zoom}, which shows the numerical solution near the boundary layer for different multilevel approximations in the absence of diffusion. 
The results indicate that the multilevel stabilization effectively damps high-frequency components of the solution, leading to a more stable approximation.
\begin{figure}[H]
    \centering
    \includegraphics[width=0.4\textwidth]{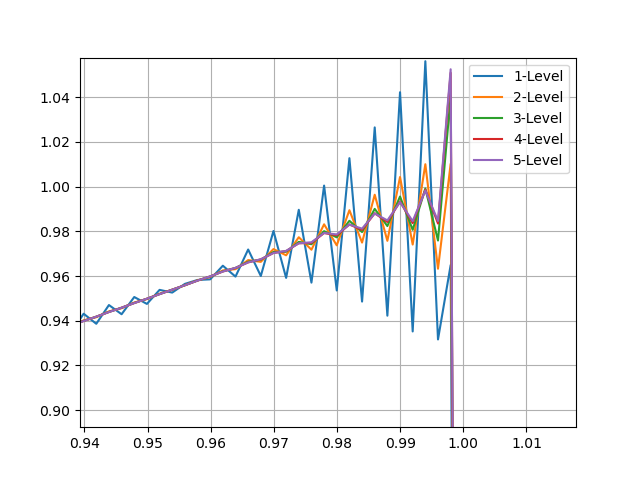}
    \caption{Zoom near the boundary layer showing solutions at different levels for $p=5$ and $n_e=512$.}
    \label{fig:boundary_layer_zoom}
\end{figure}    
\end{remark}

\begin{remark}[On the choice of the coarsest level $L$]
\label{remark:choice_of_L}
The number of levels $L$ controls the spectral breadth of the multilevel stabilization: each additional level penalizes a coarser band of unresolved scales, as confirmed by Remark~\ref{remark:multilevel_oscillations}. A natural guideline is to choose $L$ such that the coarsest mesh size $H^{(L)} = 2^L h$ remains comparable to the characteristic length scales of the data. Furthermore, this constraint on $H^{(L)}$ can be relaxed for higher polynomial degrees $p$. Because the method is weakly consistent, the enhanced approximation power of high-order splines ensures that the consistency error remains optimally bounded even on relatively coarser macro-elements. In all experiments reported here, $L \in \{3, 4, 5\}$ with $H^{(L)} < \text{diam}(\Omega)$, so that the coarsest quasi-interpolant $\Pi_L^p$ retains its full approximation order and the consistency error in Lemma~\ref{lem:consistency} is not degraded.
\end{remark}
\begin{remark}
\label{remark:condition_number}
Figure~\ref{Comp_cond} shows the condition number as a function of the number of degrees of freedom for the Galerkin, SUPG, and multilevel stabilized methods. 
The multilevel method exhibits a condition number that grows approximately linearly with respect to the number of degrees of freedom. 
In contrast, the standard Galerkin method leads to significantly larger condition numbers with a faster growth rate. 
The SUPG method also exhibits growth under refinement, but remains larger than that of the multilevel method. 

These results demonstrate that the proposed stabilization not only reduces oscillations but also improves the numerical conditioning of the discrete problem, thereby enhancing robustness and enabling the efficient use of high-degree B-splines.
\begin{figure}[H]
    \centering
    \includegraphics[width=0.4\textwidth]{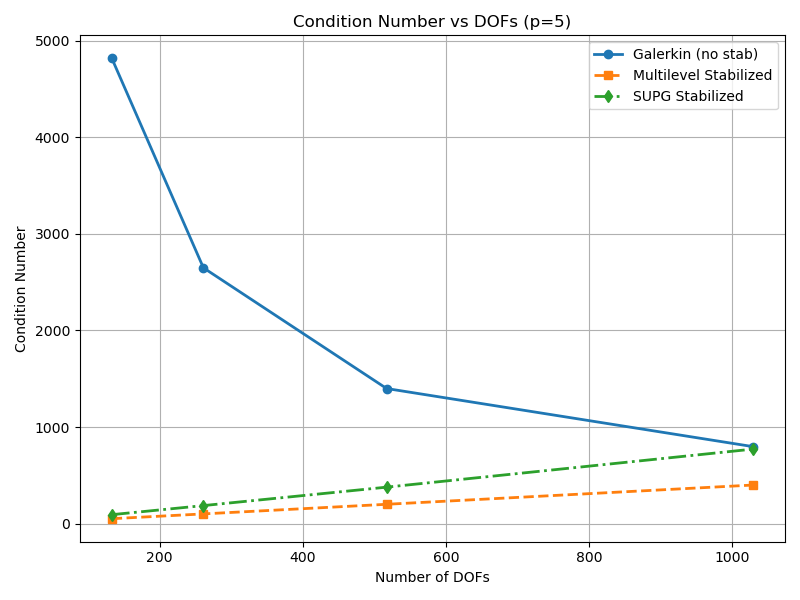}
    \caption{Condition number comparison for the 1D advection-diffusion problem \eqref{pd0} (Test 2) across Galerkin, SUPG, and Multilevel stabilized methods for polynomial degree $p=5$.}
    \label{Comp_cond}
\end{figure}   
\end{remark}

\subsection{Test 3: Problem with parabolic boundary layers}

In this example, we consider a convection--diffusion problem whose solution exhibits two parabolic boundary layers. The problem is characterized by a small diffusion parameter
\[
\varepsilon = 10^{-8},
\]
and a convection field defined by
\[
\boldsymbol{b}(x,y) = \bigl(0,\; 1 + x^2\bigr)^T,
\]
while the reaction coefficient is set to
\[
c = 0.
\]

Let $\Omega = (0,1)^2$ be the computational domain. The boundary $\partial \Omega$ is decomposed into a Neumann part and a Dirichlet part as follows:
\[
\Gamma_N := \{(x,y) \in \partial \Omega \;:\; 0 < x < 1,\; y = 1\}, 
\qquad 
\Gamma_D := \partial \Omega \setminus \Gamma_N.
\]

On the Neumann boundary $\Gamma_N$, homogeneous Neumann conditions are prescribed, i.e.,
\[
g_N = 0.
\]
On the Dirichlet boundary $\Gamma_D$, the boundary data is given by
\[
g_D(x,y) =
\begin{cases}
1, & \text{for } 0 \leq x \leq 1,\; y = 0, \\[6pt]
1 - y, & \text{otherwise}.
\end{cases}
\]

This configuration leads to the formation of parabolic boundary layers induced by the variable convection field, and serves as a relevant benchmark for assessing the robustness of stabilization methods. For the numerical experiments in this test, the multilevel parameter is set to $L=4$ when the polynomial degree is $p=3$, and $L=5$ when $p=5$.
\begin{figure}[H]
    \centering
    \includegraphics[width=1\textwidth]{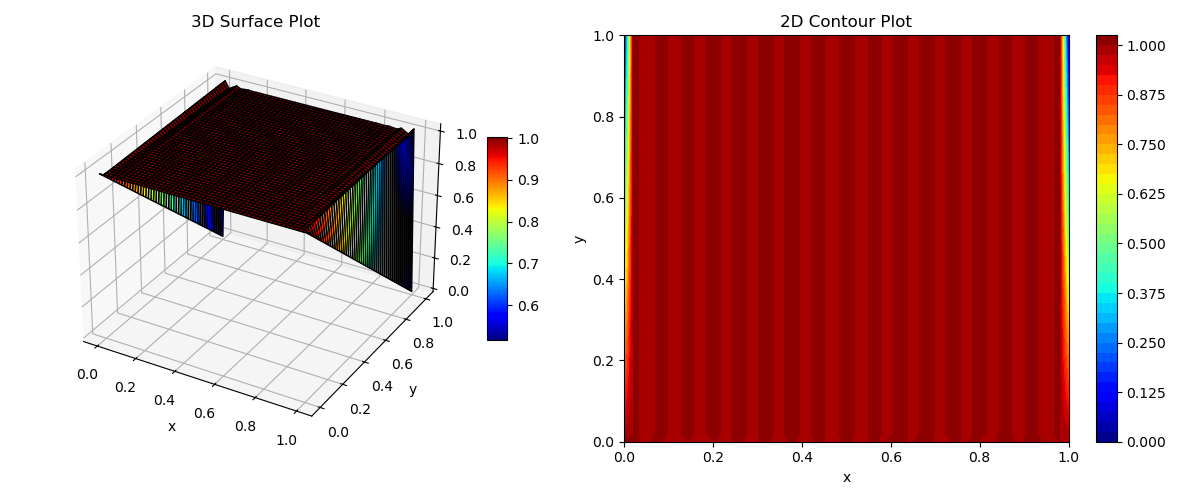}
    \caption{Standard Galerkin surface plot for the parabolic boundary layers problem \eqref{ADR_equa} (Test 3) with $p=5$ and $n_e=64 \times 64$.}
    \label{3D_parabolique_Galerkin}
\end{figure}  
\begin{figure}[H]
    \centering
    \includegraphics[width=1\textwidth]{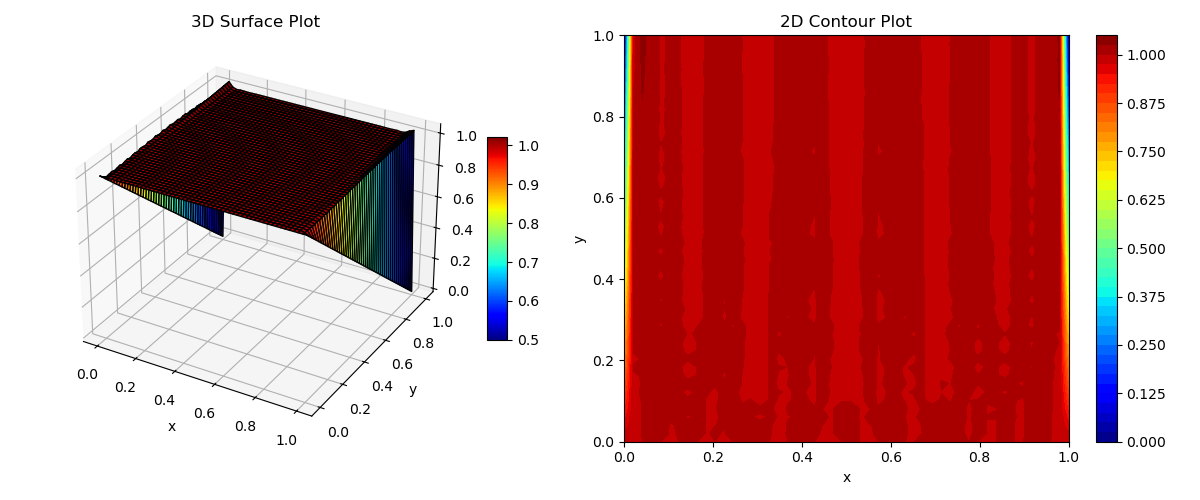}
    \caption{Multilevel stabilized surface plot for the parabolic boundary layers problem \eqref{ADR_equa} (Test 3) with $p=5$ and $n_e=64 \times 64$.}
    \label{3D_parabolique_Multilevel}
\end{figure}  
\begin{figure}[H]
    \centering
    \begin{subfigure}[b]{0.45\textwidth}
        \centering
        \includegraphics[width=\textwidth]{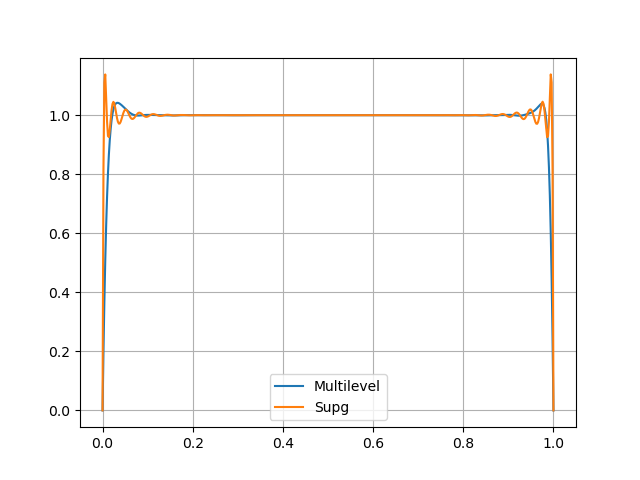} 
        \caption{Solution for $p=5$.}
        \label{fig:parabolic_layer_p5}
    \end{subfigure}
    \hfill
    \begin{subfigure}[b]{0.45\textwidth}
        \centering
        \includegraphics[width=\textwidth]{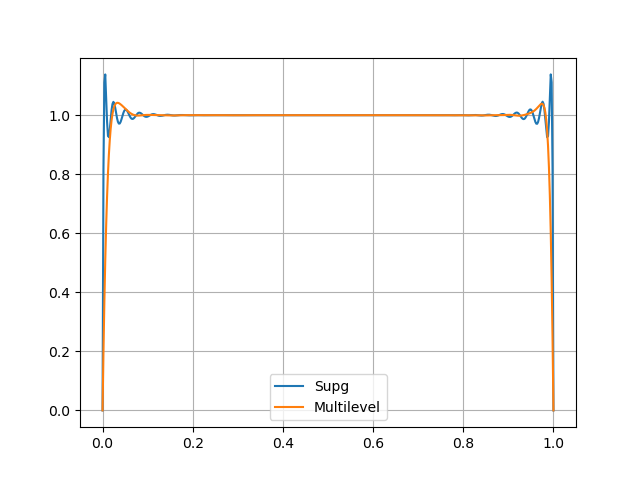} 
        \caption{Solution for $p=3$.}
        \label{parabolique_SUPG_Muiltilevel}
    \end{subfigure}

    \caption{Comparison of stabilized solutions for the parabolic boundary layers problem \eqref{ADR_equa} (Test 3) with $n_e=64 \times 64$. (a) Polynomial degree $p=5$. (b) Polynomial degree $p=3$.}
    \label{parabolique_layer_}
\end{figure}

These results (Figures~\ref{parabolique_layer_}, \ref{3D_parabolique_Multilevel}, \ref{3D_parabolique_Galerkin}) compare favorably with those reported in~\cite{matthies_stabilization_2008} for classical LPS methods, confirming the effectiveness of the proposed multilevel approach in the isogeometric setting.
\subsection{Test 4: Problem with two internal layers}
We consider problem \eqref{ADR_equa} with diffusion coefficient $\varepsilon = 10^{-5}$, convection field $\boldsymbol{b}(x,y) = (1,0)^T$, and homogeneous Dirichlet boundary conditions $u = 0$ on $\partial \Omega$. The source term is defined by
\[
f(x,y) =
\begin{cases}
16(1 - 2x), & \text{if } (x,y) \in [0.25, 0.75]^2, \\
0, & \text{otherwise}.
\end{cases}
\]

This test case, introduced in~\cite{john_diminishing_2007}, serves as a benchmark for convection-dominated problems in which linear and nonlinear discretizations often fail to provide satisfactory results. The exact solution is well approximated by the quadratic function $(4x - 1)(3 - 4x)$ within the region $[0.25, 0.75]^2$, while it remains close to zero (yet positive) elsewhere in $\Omega$. This function will be referred to as the \emph{reference solution} in the cross-sectional plots presented below.

Although the exact solution is non-negative, many nonlinear schemes yield numerical approximations that exhibit negative values in certain regions of the domain (see~\cite{john_spurious_2008} for a detailed discussion). This behavior does not contradict the discrete maximum principle (DMP), since the right-hand side $f$ changes sign inside $\Omega$, but it remains undesirable from both physical and numerical viewpoints.

Figure~\ref{3D_Example3} illustrates the surface plots obtained with different methods, while Figure~\ref{Cross_section_Example3} shows cross-sections along the lines $x = 0.5$. Along the line $x = 0.5$, SUPG and Galerkin display similar behavior.
\begin{figure}[H]
    \centering
    \begin{subfigure}[b]{0.45\textwidth}
        \centering
        \includegraphics[width=\textwidth]{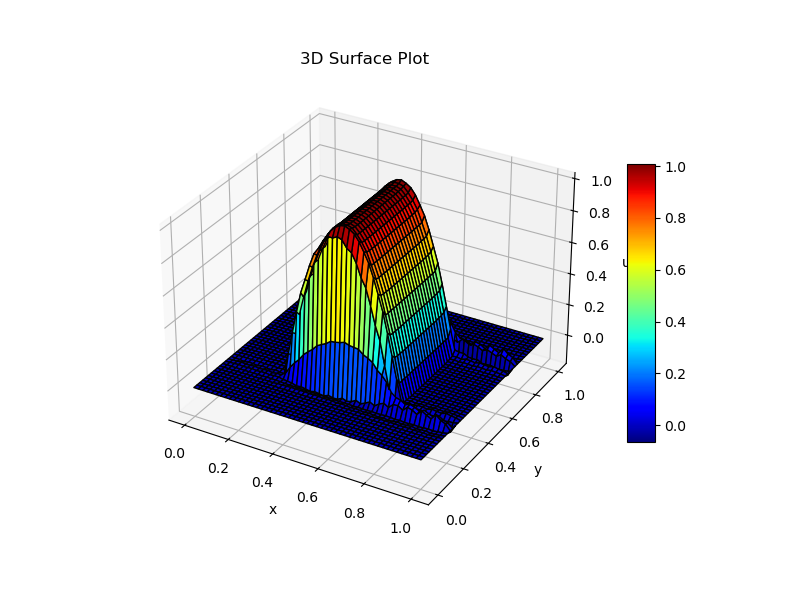} 
        \caption{Multilevel stabilized solution.}
        \label{3D_Example3_Multilevel_p=3}
    \end{subfigure}
    \hfill
    \begin{subfigure}[b]{0.45\textwidth}
        \centering
        \includegraphics[width=\textwidth]{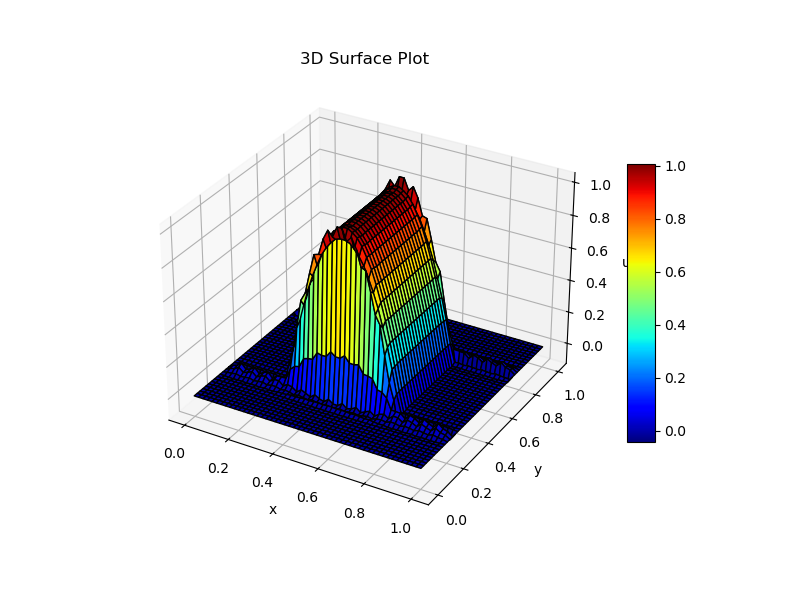} 
        \caption{Standard Galerkin solution.}
        \label{D_Example3_Galerkin_p=3}
    \end{subfigure}

    \caption{Surface plots of the numerical solutions for the two internal layers problem \eqref{ADR_equa} (Test 4) with $p=3$ and $n_e=64 \times 64$. (a) Multilevel stabilized solution. (b) Standard Galerkin solution.}
    \label{3D_Example3}
\end{figure}

\begin{figure}[H]
    \centering
    \begin{subfigure}[b]{0.45\textwidth}
        \centering
        \includegraphics[width=\textwidth]{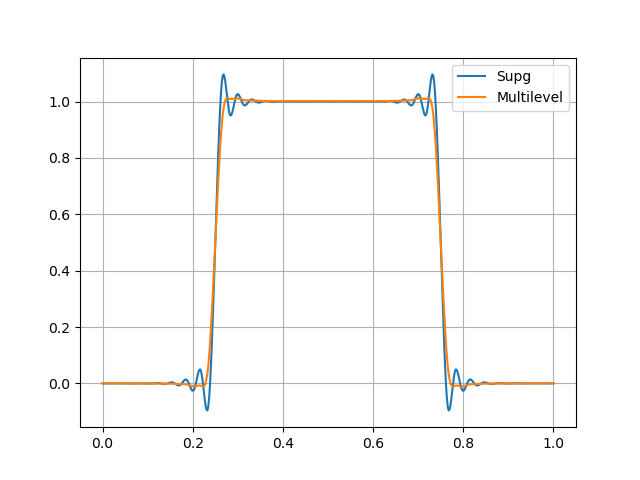} 
        \caption{Multilevel vs. SUPG.}
        \label{fig:cross_example3_multi_supg}
    \end{subfigure}
    \hfill
    \begin{subfigure}[b]{0.45\textwidth}
        \centering
        \includegraphics[width=\textwidth]{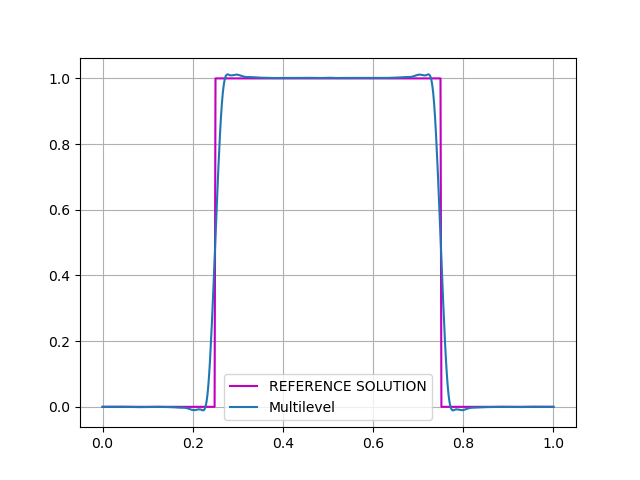} 
        \caption{Multilevel vs. reference solution.}
        \label{fig:cross_example3_multi_ref}
    \end{subfigure}

    \caption{Cross-sectional profiles at $x=0.5$ for the two internal layers problem \eqref{ADR_equa} (Test 4) with $p=3$, $n_e=64 \times 64$, and $L=5$. (a) Comparison between Multilevel and SUPG methods. (b) Comparison between Multilevel method and the reference solution.}
    \label{Cross_section_Example3}
\end{figure}
Finally, to make a more detailed analysis for this problem, we have repeated the study carried out in \cite{john_spurious_2008,barrenechea_blending_2017,john_diminishing_2007} for this example. More precisely, we have computed the quantities

\[
\min := -\min_{0.4 \leq x \leq 0.6} u_h(x, y), \quad 
\text{diff} := \max_{x \geq 0.8} u_h(x, y) - \min_{x \geq 0.8} u_h(x, y).
\]

Table~\ref{tab:example3_results} summarizes these indicators for both the standard SUPG method and the proposed Multilevel stabilization method for polynomial degrees $p=2,3,4,5$. The quantity \textit{min} measures the undershoot near the internal layers, while \textit{diff} reflects the variation of the solution. We observe that the Multilevel method consistently produces significantly smaller undershoots than SUPG for all degrees considered, with values of \textit{min} decreasing from 0.0316, 0.0283, 0.0254, to 0.0205 for $p=2,3,4,5$, respectively. This indicates improved monotonicity and enhanced stability near sharp internal layers. Notably, both the \textit{min} and \textit{diff} indicators are competitive with those reported for most nonlinear SOLD (Spurious Oscillations at Layers Diminishing) methods studied in~\cite{john_spurious_2008,john_diminishing_2007}. The key distinction is that the proposed Multilevel method achieves this level of performance as a \emph{fully linear} formulation, without requiring the nonlinear artificial viscosity terms and iterative solvers that SOLD methods rely upon. Therefore, the proposed Multilevel stabilization method maintains stability and physical consistency even for higher-order polynomial approximations, while preserving the computational simplicity of a linear scheme.
\begin{table}[H]
\centering
\caption{Comparison of stability indicators for Test 4 for different polynomial degrees $p$ ($n_e=64 \times 64$, $L=5$).}
\begin{tabular}{lccc}
\hline
Method & $p$ & min & diff \\
\hline
SUPG        & 2 & 0.0725 & 0.0332 \\
Multilevel  & 2 & 0.0316 & 0.1779 \\
\hline
SUPG        & 3 & 0.0590 & 0.2168 \\
Multilevel  & 3 & 0.0283 & 0.2906 \\
\hline
SUPG        & 4 & 0.0807 & 0.0393 \\
Multilevel  & 4 & 0.0254 & 0.2307 \\
\hline
SUPG        & 5 & 0.0464    & 0.2189    \\
Multilevel  & 5 & 0.0205 & 0.2788 \\
\hline
\end{tabular}
\label{tab:example3_results}
\end{table}

These results demonstrate that the proposed method achieves competitive or superior performance compared to the approaches reported in~\cite{barrenechea_blending_2017,john_spurious_2008}.
\subsection{Test 5: Rotational problem} 
We consider problem~\eqref{ADR_equa} on $\Omega=(0,1)^2$ with $\varepsilon = 10^{-7}$, convection field $\mathbf{b}(x,y)=(-y,x)$, reaction $c=0$, and source $f=0$. 
The inflow boundary condition is 
\[
u_b(x,0) = 
\begin{cases} 
1, & x \in \left(\frac{1}{3},\frac{2}{3}\right), \\[2mm]
0, & \text{otherwise.}
\end{cases}
\] 
At the outflow boundary $(0,1) \times \{1\}$, a homogeneous Neumann condition is imposed. 
The solution develops two interior layers originating from the discontinuities in the inflow profile at $y=0$. For this test, the multilevel parameter is $L=4$.
\begin{figure}[H]
    \centering
    \includegraphics[width=0.8\textwidth]{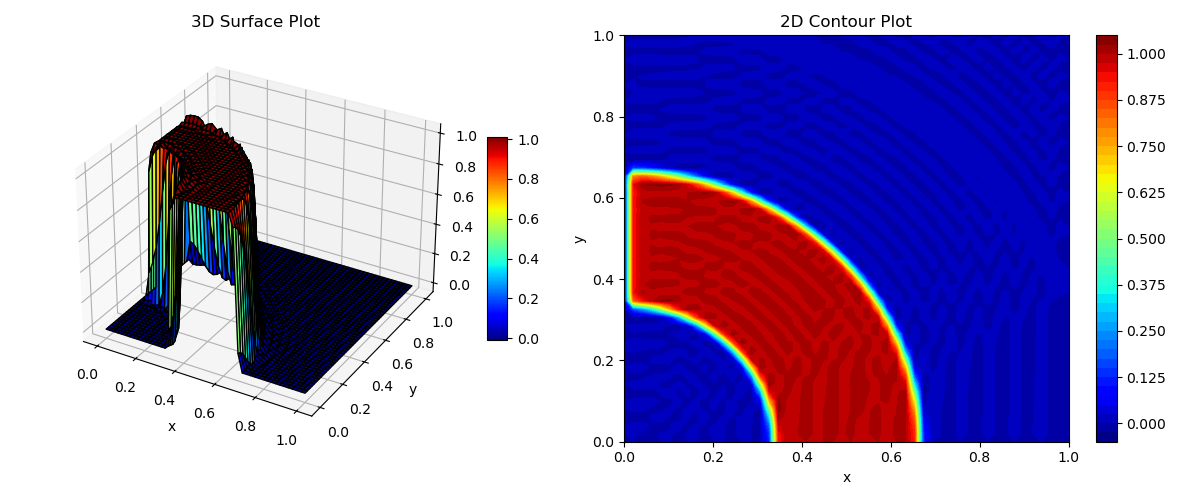}
    \caption{SUPG surface plot for the rotational problem \eqref{ADR_equa} (Test 5) with $p=3$ and $n_e=64 \times 64$.}
    \label{3D_Rotation_SUPG}
\end{figure}  
\begin{figure}[H]
    \centering
    \includegraphics[width=0.8\textwidth]{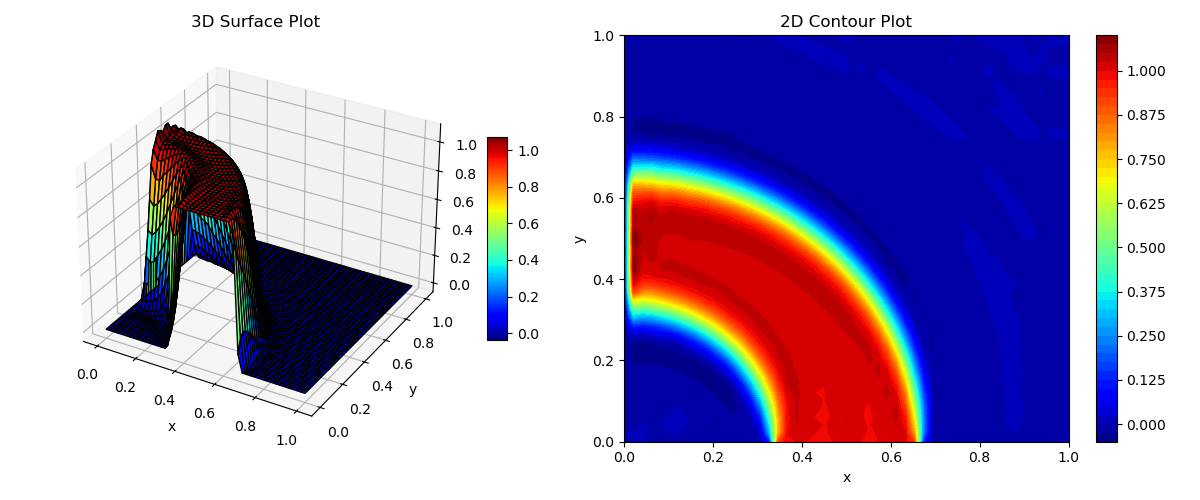}
    \caption{Multilevel stabilized surface plot for the rotational problem \eqref{ADR_equa} (Test 5) with $p=3$ and $n_e=64 \times 64$.}
    \label{3D_Rotation_Multilevel}
\end{figure}  
\subsection{Test 6: Two-dimensional pure advection problem}\label{sec:2d-pure-advection}

We consider the following two-dimensional pure advection problem:
\begin{equation}
\begin{cases}
\partial_{y} u = \frac{1}{2 \epsilon}\left(1 - \tanh^2\left(\frac{1}{\epsilon}(y-a)\right)\right), & \text{in } (0,1)^2, \vspace{0.2cm} \\
u_{\mid y=0} = 0, & \text{for } x \in (0,1),
\end{cases}
\label{eq:2d_pure_adv}
\end{equation}
with $\epsilon = 0.004$, whose exact solution is given by
\[
u(y) = \frac{1}{2}\left(\tanh\left(\frac{1}{\epsilon}(y-a)\right) + 1\right),
\]
where $a=\frac{2}{3}$.

The parameter $\epsilon$ controls the sharpness of the internal layer located at $y=a$. The value $\epsilon = 0.004$ adopted here produces a significantly sharper transition than the standard choice $\epsilon = 0.04$ commonly used in the literature~\cite{guermond_transport_1999,garg2023local}, thereby providing a more challenging benchmark for the stabilization method.
Figure~\ref{Pure_cross_section} shows a projection of the solution onto the plane \(x = 0\), comparing the exact solution with numerical approximations obtained using the standard Galerkin method and the proposed multilevel stabilization method.  
The results clearly illustrate the stabilizing effect of the multilevel approach, which effectively reduces spurious oscillations while accurately capturing the sharp layer at \(y = a\).
\begin{figure}[H]
    \centering
    \begin{subfigure}[b]{0.45\textwidth}
        \centering
        \includegraphics[width=\textwidth]{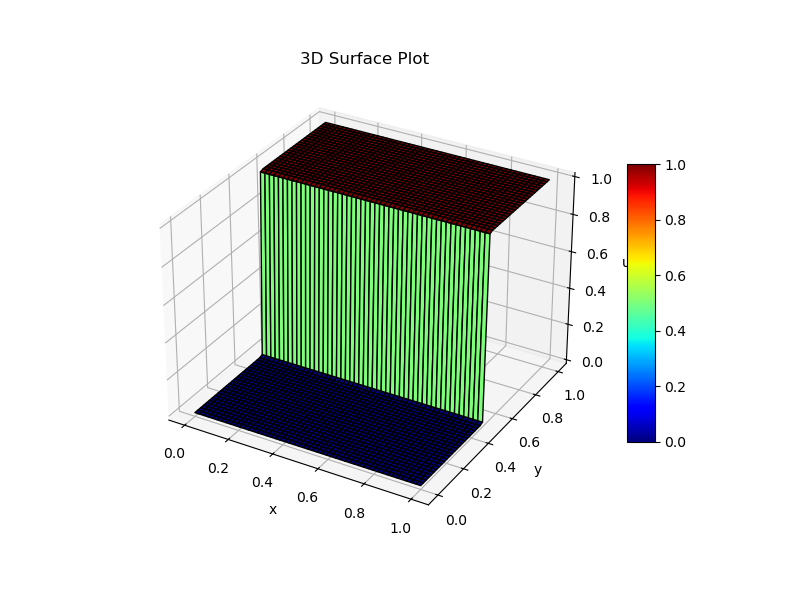} 
        \caption{Exact solution.}
        \label{3D_Pure_Multilevel_p=3}
    \end{subfigure}
    \hfill
    \begin{subfigure}[b]{0.45\textwidth}
        \centering
        \includegraphics[width=\textwidth]{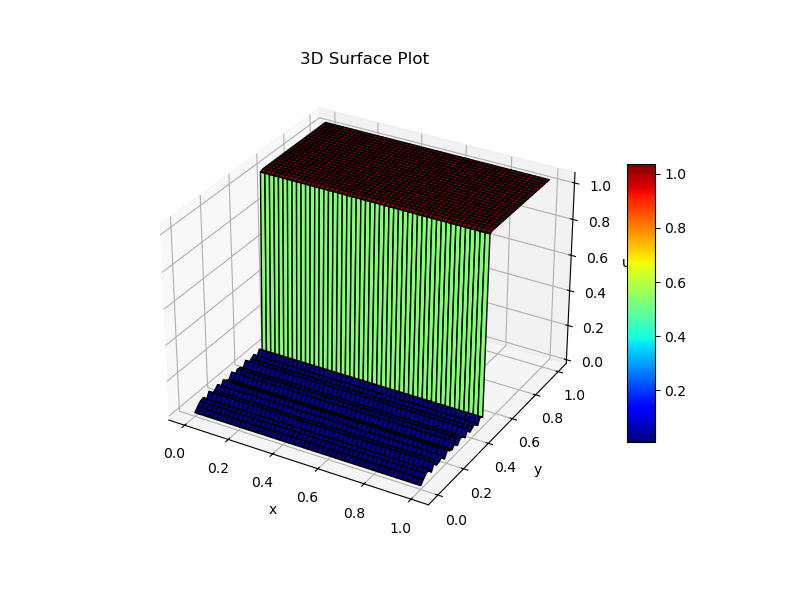} 
        \caption{Galerkin solution ($\mathbf{p}=(2,3)$).}
        \label{3D_pure_Galerkin_p=3}
    \end{subfigure}

    \caption{Surface plots for the 2D pure advection problem \eqref{eq:2d_pure_adv} (Test 6) using $n_e=(8,256)$. (a) Exact solution. (b) Standard Galerkin solution with polynomial degrees $\mathbf{p}=(2,3)$.}
    \label{3D_pure_advec}
\end{figure}
\begin{figure}[H]
    \centering
    \begin{subfigure}[b]{0.45\textwidth}
        \centering
        \includegraphics[width=\textwidth]{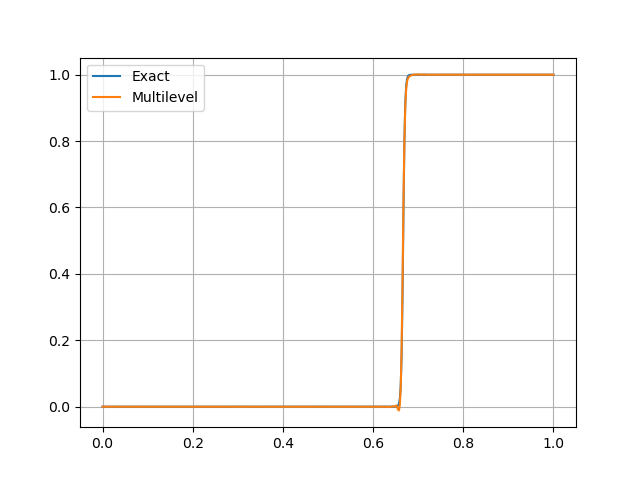} 
        \caption{Multilevel stabilized solution.}
        \label{Pure_Multilevel_p=5}
    \end{subfigure}
    \hfill
    \begin{subfigure}[b]{0.45\textwidth}
        \centering
        \includegraphics[width=\textwidth]{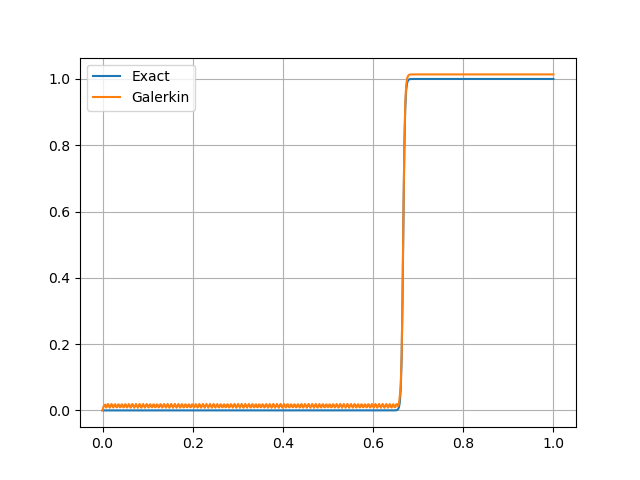} 
        \caption{Standard Galerkin solution.}
        \label{pure_Galerkin_p=5} 
    \end{subfigure}

    \caption{Cross-sectional profiles for the 2D pure advection problem \eqref{eq:2d_pure_adv} (Test 6) with polynomial degrees $\mathbf{p}=(2,5)$ and mesh size $n_e=(8,256)$. (a) Multilevel stabilized solution. (b) Standard Galerkin solution.}
    \label{Pure_cross_section}
\end{figure}

Table \ref{Table:p3_relative} presents the relative $L^2$ and $H^1$ errors with convergence rates for the 2D pure advection problem using both the standard Galerkin and the proposed Multilevel method with 
p=(2,3). The Multilevel method clearly outperforms Galerkin, achieving much lower errors even on coarse meshes and stable convergence on finer meshes. The $L^2$ error decreases smoothly, while the $H^1$ error remains reasonable despite the sharp layer in the y-direction. In contrast, Galerkin exhibits large errors and very slow convergence on coarse meshes. Note that the unusually high convergence rates observed between $n_{e,2}=256$ and $512$ (e.g., $r \approx 5.12$ for $p=3$) are characteristic of pre-asymptotic behavior. Because the internal layer is exceptionally sharp ($\epsilon = 0.004$), the coarser meshes severely under-resolve the gradient; once the mesh size $h$ becomes smaller than $\epsilon$, the layer is fully resolved, leading to a sudden drop in the error before it settles into the optimal asymptotic rate. Crucially, the Multilevel method's convergence rate is less inflated than Galerkin's ($5.12$ vs $6.29$) precisely because the proposed stabilization successfully suppressed the massive coarse-mesh errors that standard Galerkin suffers from. These results demonstrate the robustness and efficiency of the Multilevel stabilization for pure advection problems.
\begin{table}[H]
\caption{Relative convergence test with rates for $\mathbf{p}=(2,3)$ of 2D pure advection \eqref{eq:2d_pure_adv} by Galerkin and Multilevel method.} 
\label{Table:p3_relative}
\centering
\begin{tabular}{|c|c|c|c|c|c|c|c|c|}
\hline
$n_{e,2}$ & L$^2$-G & R-G & H$^1$-G & R-G & L$^2$-ML & R-ML & H$^1$-ML & R-ML \\
\hline
64  & 0.3920   & -    & 1.6015   & -    & 0.0455   & -    & 0.6222   & -    \\
128 & 0.1985   & 0.98 & 1.5507   & 0.05 & 0.0240   & 0.92 & 0.4992   & 0.32 \\
256 & 0.0314   & 2.66 & 0.4791   & 1.69 & 0.00734  & 1.71 & 0.1645   & 1.60 \\
512 & 0.000402 & 6.29 & 0.0127   & 5.23 & 0.000211 & 5.12 & 0.00808  & 4.35 \\
\hline
\end{tabular}
\end{table}
\begin{figure}[H]
    \centering
    \begin{subfigure}[b]{0.45\textwidth}
        \centering
        \includegraphics[width=\textwidth]{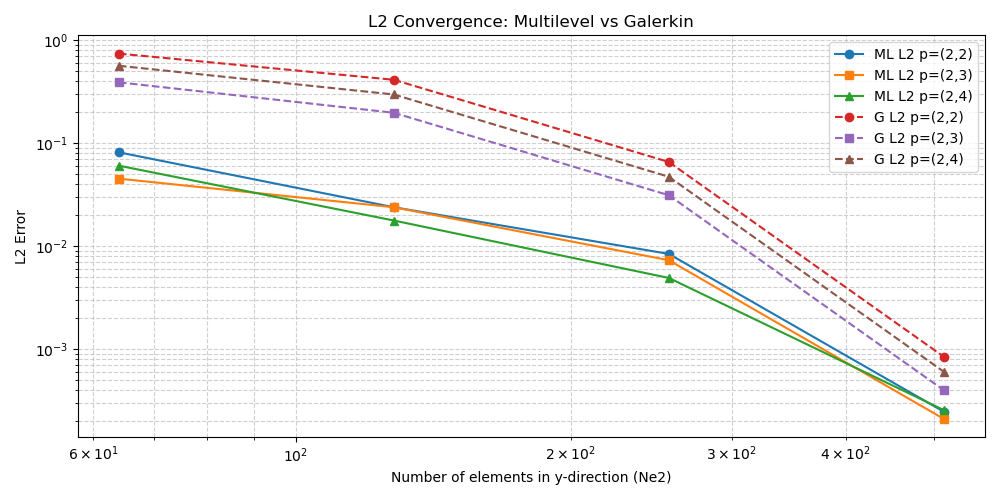} 
        \caption{$L^2$-norm convergence.}
        \label{Pure_L2_cv}
    \end{subfigure}
    \hfill
    \begin{subfigure}[b]{0.45\textwidth}
        \centering
        \includegraphics[width=\textwidth]{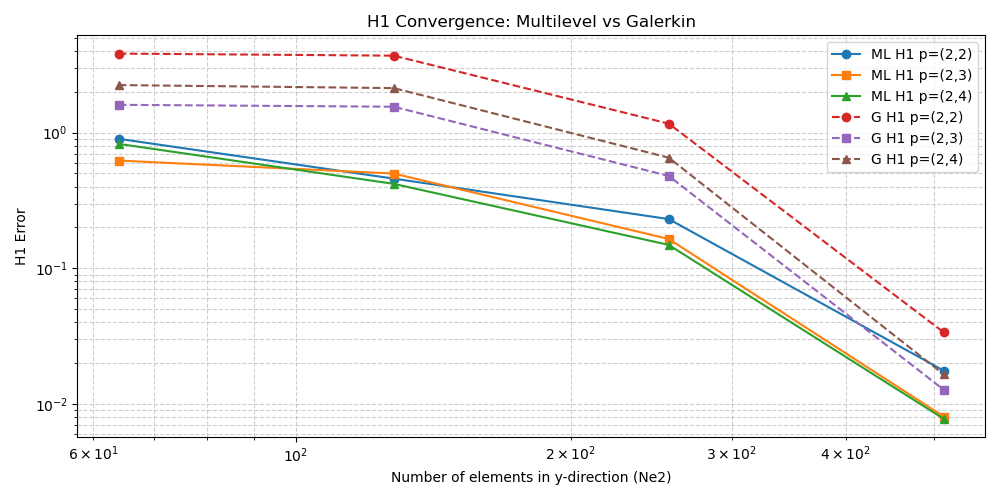} 
        \caption{$H^1$-norm convergence.}
        \label{Pure_H1_cv} 
    \end{subfigure}

\caption{Relative error convergence plots for the 2D pure advection problem \eqref{eq:2d_pure_adv} (Test 6) with different polynomial degrees. (a) Relative $L^2$-norm error. (b) Relative $H^1$-norm error.}
\label{Pure_cv_plots}
\end{figure}

As reported in~\cite{guermond_transport_1999}, the linear system associated with the standard Galerkin approximation is severely ill-conditioned, to the extent that the introduction of a small regularization term was necessary to obtain a stable solution. In contrast, the stabilized solutions do not require such a regularization term, and the linear systems arising from our method remain well conditioned, as illustrated in Figure~\ref{Condition_number_pure}.

\begin{figure}[H]
    \centering
    \includegraphics[width=0.50\textwidth]{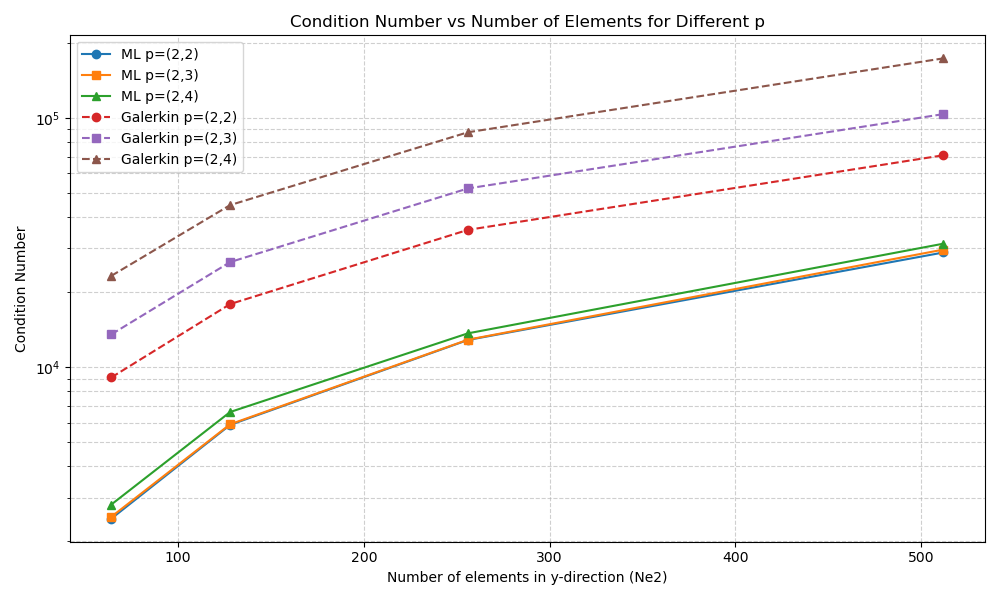}
    \caption{Condition number comparison for the pure advection problem across different polynomial degrees.}
    \label{Condition_number_pure}
\end{figure}
\section{Conclusions}

We have presented a novel Multilevel Isogeometric Quasi-Interpolant (MQ) projection method for the stabilization of steady advection-dominated and advection--reaction problems. The core mechanism of the method is the hierarchical penalization of fine-scale fluctuations using continuous B-spline quasi-interpolants across nested levels of the discrete space. By leveraging the $C^{p-1}$ continuity of B-splines, the proposed approach avoids both the discontinuous auxiliary spaces required by Local Projection Stabilization and the strong PDE residual evaluations required by SUPG and VMS methods. 

First, the analysis of the discrete method established the theoretical foundation by proving optimal \emph{a priori} error estimates in the MQ norm. A key theoretical contribution is the explicit capture of the multilevel hierarchy's effect through the weighted sum $\Sigma(p, L)$, which demonstrates that the stabilization is natively valid for arbitrary polynomial degrees $p$. Furthermore, we introduced the MQSD norm to establish a discrete inf-sup condition for the 1D setting under a numerically validated stability hypothesis, confirming that the method provides explicit control over the streamline derivative. While the rigorous proof is currently restricted to the 1D constant advection case, the extensive multidimensional numerical results provide strong evidence supporting its generalization; a complete functional analysis proof for the general setting remains an important open problem. 

Numerically, the method was evaluated on a stringent suite of benchmarks, including internal layers and parabolic boundary layers at high Péclet numbers. The multilevel projection consistently outperformed standard Galerkin and SUPG formulations, effectively suppressing spurious oscillations. Notably, despite being a fully linear formulation, the method demonstrated significant reduction of undershoots near sharp layers. In particular, the undershoot and outflow-variation indicators on the challenging two-internal-layers benchmark are competitive with those reported for most nonlinear SOLD methods~\cite{john_spurious_2008,john_diminishing_2007}, despite the fully linear character of the formulation.

While the present results are highly encouraging, several open challenges remain. First, a functional analysis proof of the discrete inf-sup condition for variable convection fields in multidimensions remains an open mathematical problem. Another important direction is generalizing these results to the multidimensional case in the context of general geometric mappings. Finally, extending this methodology to time-dependent transport phenomena, potentially through integration within a space-time isogeometric framework, and the incompressible Navier-Stokes equations constitutes a promising direction for future research.
\bibliographystyle{unsrt}  
\bibliography{Multilevel_ref}

\end{document}